\documentclass[leqno,11pt, a4]{amsart}
\usepackage[usenames,dvipsnames,svgnames,table]{xcolor}

\usepackage[active]{srcltx}
\usepackage{pb-diagram}

\usepackage{dsfont}

\usepackage{mathrsfs}
\usepackage{amsmath}
\usepackage{amssymb}
\usepackage{amscd}
\usepackage{amsthm}
\usepackage[latin1]{inputenc}
\usepackage{graphics}
\usepackage{varioref}

\usepackage[latin1]{inputenc}
\usepackage{graphics}
\usepackage{pdfsync}

\labelformat{enumi}{(#1)}

\newtheorem{teo}[equation]{Theorem}
\newtheorem{defin}[equation]{Definition}
\newtheorem{remark}[equation]{Remark}
\newtheorem{prop}[equation]{Proposition}
\newtheorem{cor}[equation]{Corollary}
\newtheorem{lemma}[equation]{Lemma}
\newtheorem{ese}[equation]{Example}

\newtheoremstyle{named}{}{}{\itshape}{}{\bfseries}{.}{.5em}{\thmnote{#3}#1}
\theoremstyle{named}

\newcommand{\meno}{^{-1}}

\newcommand{\misu}{\mathscr{M}}

\newcommand{\liu}{\mathfrak{u}}
\newcommand{\liek}{\mathfrak{k}}

\newcommand{\lieg}{\mathfrak{g}}

\newcommand{\liep}{\mathfrak{p}}

\newcommand{\lia}{\mathfrak{a}}
\newcommand{\supp}{\operatorname{supp}}

\newcommand{\grad}{\operatorname{grad}}

\newcommand{\vacuo}{\emptyset}

\newcommand{\enf}{\emph}

\newcommand{\desudt}[1] []      {\dfrac {\mathrm {d} #1 }{\mathrm {dt}}}
\newcommand{\desudtzero}        {\desudt \bigg \vert _{t=0} }
\newcommand{\deze}        {\desudt \bigg \vert _{t=0} }

\newcommand{\desus}{\dfrac  {\mathrm{d}}{\mathrm {ds}} \bigg \vert _{s=0} }
\newcommand{\deriva}[2]{\dfrac {\mathrm{d}}{\mathrm {d#1}} \bigg \vert _{#2}}

\newcommand{\Ad}{\operatorname{Ad}}

\newcommand{\sx}{\langle}
\newcommand{\xs}{\rangle}
\newcommand{\scalo}{\sx \cdot , \cdot \xs}
\newcommand{\relint}{\operatorname{relint}}

\newcommand{\cd}{\cdot}
\renewcommand{\setminus}{-}

\newcommand{\ra}{\rightarrow}
\newcommand{\lra}{\longrightarrow}
\newcommand{\C}{\mathbb{C}}
\newcommand{\R}{\mathds{R}}

\newcommand{\om}{\omega}

\renewcommand{\phi}{\varphi}

\newcommand{\fun}{\mathfrak{F}}

\newcommand{\spaz}{\mathscr{M}}
\newcommand{\mume}{\fun\meno(0)}

\newcommand{\x}{{v}}

\newcommand{\meo}{\end{document}}

\newcommand{\campo}{{\funp}^{\#}}
\newcommand{\funp}{\tilde \fun_\liep}
\begin{document}
\title[Convexity properties of gradient maps]{Convexity properties of gradient maps associated to real reductive representations}
\author{Leonardo Biliotti}
\address{(Leonardo Biliotti) Dipartimento di Scienze Matematiche, Fisiche e Informatiche \\
          Universit\`a di Parma (Italy)}
\email{leonardo.biliotti@unipr.it}
\begin{abstract}
Let $G$ be a connected real reductive Lie group  acting linearly on a finite dimensional vector space $V$ over $\R$. This action admits a Kempf-Ness function and so we have an associated gradient map. If $G$ is Abelian we explicitly compute the image of $G$ orbits under the gradient map,  generalizing a result proved by Kac and Peterson \cite{kacp}. A similar result  is proved for the gradient map associated to the natural $G$ action on $\mathbb P(V)$. We also investigate the convex hull of the image of the gradient map restricted on the closure of $G$ orbits. Finally, we give a new proof of the Hilbert-Mumford criterion for real reductive Lie groups avoiding any algebraic result.
\end{abstract}
 \keywords{Gradient maps; real reductive representations, real reductive Lie groups, geometric invariant theory}

%
 \subjclass[2010]{22E45,53D20; 14L24}
\thanks{The author was partially supported by PRIN  2015
   ``Variet\`a reali e complesse: geometria, topologia e analisi armonica'' and GNSAGA INdAM. }
\maketitle
\section{Introduction}
Let $U$ be a compact connected Lie group and let $U^\C$ be its
  complexification.  Let $(Z,\omega)$ be a K\"ahler manifold on which
  $U^\C$ acts holomorphically. Assume that $U$ acts in a Hamiltonian
  fashion with momentum map $\mu:Z \lra\liu^*$.  This means that $\om$
  is $U$-invariant, $\mu$ is $U$-equivariant and for any $\beta \in \liu$
  we have $
    d\mu^\beta = i_{\beta_Z} \om,$
  where $\mu^\beta (x)= \mu (x)(\beta)$ and $\beta_Z$ denotes the
  fundamental vector field on $Z$ induced by the action of $U$.  It is
  well-known that the momentum map represents a fundamental tool in
  the study of the action of $U^\C$ on $Z$. Of particular importance
  are convexity theorems \cite{atiyah-commuting,guillemin-sternberg-convexity-1,kirwan}, which depend on the fact
  that the functions $\mu^\beta$ are Morse-Bott with even indices. Assume that $U$ is a compact torus.  If $Z$ is compact,  then Atiyah proved a convexity Theorem along $U^\C$ orbits \cite{atiyah-commuting}.  Recently, Biliotti and Ghigi \cite{bilio-ghigipr} proved a convexity Theorem along orbits in a very general setting using only so-called Kempf-Ness function.
The original setting for Kempf-Ness function is the following: let $V$ be a unitary representation of $U$. For $x = [v] \in \mathbb P (V)$ and $g\in U^\C$ we set $\Psi(x,g)=\log \frac{\parallel g v \parallel}{\parallel v \parallel }$  \cite{kempf-ness}. The behavior of the corresponding gradient map is encoded in the $\Psi$.  In 1990 Richardson and Slodowoy \cite{rs} proved that the Kempf-Ness Theorem extends to the case of real reductive representations. This pioneering work has allowed to prove many results exploiting tools from geometric invariant theory.  This is the perspective taken, amongst many others, in the papers \cite{berlein-jablonki,lau1,lau2}.
Recently  Der\'e and Lauret \cite{dl} use nice convexity properties of the moment map for the variety of nilpotent Lie algebras to investigate which nilpotent Lie algebras admit a Ricci negative solvable extension.  This motivated us to investigate convexity properties of gradient maps associated  to real reductive representations. We point out that  there exist several non equivalent definitions of real reductive Lie group in the literature (\cite{borel,harish-chandra,knapp-beyond,wallach}). Since we are interested in real reductive representations,  we restricted ourselves to linear groups, i.e., subgroups of $\mathrm{GL}(V)$, where $V$ is a finite dimensional real vector space. By a Theorem of Mostow \cite{mostow-self}, if $G\subset \mathrm{GL}(n,\R)$ is closed under transpose then it is compatible with respect to the Cartan decomposition of $\mathrm{GL}(V)$. Hence we fix the following setup.

Let $\rho:G \lra \mathrm{GL}(V)$ be a faithful representation on a finite dimensional real vector space. We identify $G$ with  $\rho(G)\subset \mathrm{GL}(V)$ and we assume that $G$ is closed and it is closed under transpose. This means  there exists a scalar product $\scalo$ on $V$ such that
$
G=K\exp (\liep),
$
where $K=G\cap \mathrm{O}(V)$ and $\liep=\lieg \cap \mathrm{Sym}(V)$. Here we denote by $\mathrm{O}(V)$ the orthogonal group with respect to $\scalo$, by $\mathrm{Sym}(V)$ the set of symmetric endomorphisms of $V$ and finally with $\lieg$ the Lie algebra of $G$. Then $\lieg=\mathfrak k \oplus \liep$ is the Cartan decomposition, that is $[\mathfrak k, \mathfrak k ]\subset \mathfrak k$, $[\mathfrak k, \mathfrak p ]\subset \mathfrak p$ and $[\mathfrak p, \mathfrak p ]\subset \mathfrak k$. Moreover, $K$ is a maximal compact subgroup of $G$, the map
$
K \times \liep \mapsto G$, $(k,\xi) \mapsto k\exp(\xi),
$
is a diffeomorphism, any two maximal Abelian subalgebras of $\liep$ are conjugate by an element of $K$ and the decomposition $G=KTK$ holds, where $T=\exp(\mathfrak t)$ is the connected Abelian subgroup corresponding to a maximal Abelian subalgebra $\mathfrak t$ contained in $\liep$ \cite{helgason,knapp-beyond}. In this setting,  the function
\[
\Psi:G\times V \lra \R, \qquad (g,x) \mapsto \frac{1}{2}(\langle g x, gx \rangle - \langle  x , x \rangle ).
\]
is a Kempf-Ness function (see section \ref{sec:abstract-setting}) and the corresponding gradient map is given by
\[
\fun_\liep :V \lra \liep^*, \qquad \fun_\liep (x) (\xi)=\langle \xi x, x \rangle.
\]
If $\lia\subset \liep$ is an Abelian subalgebra, then $\Psi_{|_{A\times V}}$ is a Kempf-Ness function with respect to the $A=\exp(\lia)$ action on $V$ and the corresponding gradient map is given by $\fun_\lia (x)=\fun_\liep (x)_{|_{\lia}}$.
The Kempf-Ness Theorem provides geometric criterion for the closedness  of orbits of a representation of a real reductive Lie group and the existence of quotient \cite{berlein-jablonki,heinzner-schwarz-Cartan,kempf-ness,mundet-Crelles,rs,schwartz}. We point out that recently B\"ohm and Lafuente \cite{fb} proved the Kempf-Ness Theorem for linear actions of real reductive Lie groups avoiding any deep algebraic result.  The basic tools are the notions of stable, semistable and polystable points. Biliotti, Ghigi, Raffero and Zedda, \cite{bgs,bilio-zedda,biliotti-raffero,bilio-ghigipr}, see also \cite{mundet-Crelles,mundet-Trans,teleman-symplectic-stability}, identify an abstract setting  to develop the geometrical invariant theory for actions of real reductive Lie groups and give numerical criteria for stability, semistability and polystability. These results have been applied for actions of real reductive groups on the set of probability measure of a compact Riemmannian manifold. This problem is motivated by an application to upperbounds for the first eigenvalue of the Laplacian acting on functions \cite{biliotti-ghigi-AIF,biliotti-ghigi-American,bourguignon-li-yau,hersch}. In this paper, using the properties of the Kemf-Ness functions, we explicitly compute the image of the gradient map corresponding
to the $A$ action on $V$, respectively to the A action on $\mathbb P(V)$, along $A$ orbits (Theorems \ref{orbita} and \ref{image}) generalizing a result due to Kac and Peterson in the complex setting \cite{kacp}, see also \cite{vergne}. Roughly speaking, if $A$ acts linearly on $V$, respectively on $\mathbb P(V)$, then $\overline{A\cdot x}$ is homeomorphic to a polyhedral $C$, respectively to a polytope $P$, and any  $A$ orbit contained in $\overline{A\cdot x}$ is diffeomorphic  to the relative interior of a face of $C$, respectively of a face of $P$. As an application we obtain the Hilbert-Mumford criterion and the algebraicity of the null cone for Abelian groups (Theorems  \ref{hm} and \ref{null-cone-abelian}). We also prove a convexity Theorem of the gradient map, associated to  $A$, restricted on the closure of $G$ orbits (Theorems \ref{convex-orbit-general} and \ref{image}). Applying results proved in \cite{biliotti-ghigi-heinzner-1-preprint,biliotti-ghigi-heinzner-2,bgh-israel-p}, we completely describe the convex hull of the image of the gradient map,  with respect to $G$, restricted on the closure of $G$ orbits (Theorems \ref{convo}, \ref{convex-orbit-general-projective}).  Finally, using in a different context original ideas from \cite{georgula}, which are themselves of some interest, we give a probably new proof of the Hilbert-Mumford  criterion for real reductive groups (Theorem \ref{hmr}) avoiding any algebraic result.

This paper is organized as follows.

In the second section we recall basic notions of convex geometry. In particular we recall the definition of polyhedral, polytope, extremal points and exposed faces.
In the third section we recall the abstract setting on which we are able to develop a geometrical invariant theory for actions of real reductive Lie groups. We also recall
the notions of stable, semistable and polystable points. In the fourth section we consider a closed subgroup $G$ of $\mathrm{GL}(V)$, where $V$ is a finite dimensional real vector space endowed by a scalar product $\scalo$, which is also closed under transpose.
Given an Abelian subalgebra $\lia\subset \liep$, we explicitly compute the image of the gradient map, with respect to $A=\exp(\lia)$, along $A$ orbits and restricted on the closure of $G$ orbits.  As an application we get the Hilbert-Mumford criterion for Abelian groups and the algebraicity of the null cone. In the fifth section we completely describe the convex hull of the image of the gradient map, with respect to $G$,  restricted on the closure of $G$ orbits. We also discuss the Hilbert-Mumford criterion for real reductive groups. In the sixth section, we investigate the $G$ action on $\mathbb P(V)$.
We prove convexity Theorems for the gradient map associated to $A$ and we describe the convex hull of the image of the gradient map, with respect to $G$, restricted on the closure of $G$ orbits. 
In the last section we give a proof of the Hilbert-Mumford criterion for real reductive groups.
\section{Convex geometry}\label{convex}
It is useful to recall a few definitions and results regarding convex
sets. A good references, amongst many other, are  \cite{schneider-convex-bodies,schr} (see also  \cite{biliotti-ghigi-heinzner-1-preprint, biliotti-ghigi-heinzner-2,bgh-israel-p,villa}).

 Let $V$ be a real vector
 space with a scalar product $\scalo$ and let $E\subset V$ be a
 convex subset.  The \emph{relative interior} of
 $E$, denoted $\relint E$, is the interior of $E$ in its affine hull.
 A face $F$ of $E$ is a convex subset $F\subset E$ with the following
 property: if $x,y\in E$ and $\relint[x,y]\cap F\neq \vacuo$, then
 $[x,y]\subset F$.  The \emph{extreme points} of $E$ are the points
 $x\in E$ such that $\{x\}$ is a face. If  $E$ is closed and nonempty then the faces
 are closed \cite[p. 62]{schneider-convex-bodies}.  A face distinct
 from $E$ and $\vacuo$ will be called a \enf{proper face}.  The
 \enf{support function} of $E$ is the function $ h_E : V \ra \R$, $
 h_E(u) = \max_{x \in E} \sx x, u \xs$.  If $ u \neq 0$, the
 hyperplane $H(E, u) : = \{ x\in E : \sx x, u \xs = h_E(u)\}$ is
 called the \enf{supporting hyperplane} of $E$ for $u$. The set
   \begin{gather}
     \label{def-exposed}
     F_u (E) : = E \cap H(E,u)
   \end{gather}
   is a face and it is called the \enf{exposed face} of $E$ defined by
   $u$. 
 In general not all faces of a convex subset are exposed.
 A simple example is given by the convex hull of a closed disc and a
 point outside the disc: the resulting convex set is the union of the
 disc and a triangle. The two vertices of the triangle that lie on the
 boundary of the disc are non-exposed 0-faces. Another example is given in \cite{bgh-israel-p} $p.432$.

A subset $E\subset V$ is called a \emph{convex cone} if $E$ is convex, not empty and closed under multiplication by non negative real numbers. It is easy to check that $E$ is a convex cone if and only if $E$ is closed under addition and under multiplication by non negative real numbers. The cone generated by the vectors $f_1,\ldots,f_n \in V$ is the set $
C(f_1,\ldots,f_n):=\{\lambda_1 x_1 + \cdots +\lambda_n x_n :\, \lambda_1\geq 0,\,\ldots,\,\lambda_n \geq 0\}.
$
A cone arising in this way is called \emph{finitely generated cone}. A \emph{polytope} is the convex hull of a finite number of points of $V$. If $f_1,\ldots,f_n\in V$ then  the set
$P(f_1,\ldots,f_n)=\{\alpha_1 f_1 + \cdots +\alpha_n f_n:\, \alpha_1,\ldots,\alpha_n \geq 0$ and $\alpha_1 + \cdots + \alpha_n =1\}$ is the polytope generated by $f_1,\ldots,f_n$. The following result goes back by Farkas, Minkowsky and Weyl (see \cite{schr} $p.87$ for a proof).
\begin{teo}
A convex cone is finitely generated if and only is the intersection of finitely many closed linear half spaces.
\end{teo}
The above theorem implies $C(f_1,\ldots,f_n)$ is closed. In the literature the intersection of finitely many closed linear half spaces is called polyhedral. Hence $C(f_1,\ldots,f_n)$ is a polyhedral. The concept of polytope and polyhedral are related and this statement is usually attribute to Minkowski (see \cite{schr} p. $89$).
\begin{teo}
A convex set is a polytope if and only if it is a bounded polyhedral.
\end{teo}
If $f_1,\ldots,f_n \in V$, we denote by $C^o(f_1,\ldots,f_n)=\{\lambda_1 x_1 + \cdots +\lambda_n x_n :\, \lambda_1> 0,\,\ldots,\,\lambda_n > 0\}.$
\begin{lemma}\label{cono-chiuso}
$C^o(f_1,\ldots,f_n)$ is closed if and only if $0\in C^o(f_1,\ldots,f_n)$.
\end{lemma}
\begin{proof}
If $0\in C^o(f_1,\ldots,f_n)$, then there exist $\alpha_1,\ldots,\alpha_n >0$ such that
\[
0=\alpha_1 f_1 + \cdots + \alpha_n f_n.
\]
Let $v\in C(f_1,\ldots,f_n)$. Since $v=v+\alpha_1 f_1 + \cdots + \alpha_n f_n$ it follows $v\in  C^o(f_1,\ldots,f_n)$ and so $C^o(f_1,\ldots,f_n)=C(f_1,\ldots,f_n)$ is closed. Vice-versa, assume $C^o (f_1,\ldots,f_n)$ is closed.  Then $C^o(f_1,\ldots,f_n)=C(f_1,\ldots,f_n)$ and so $0\in C^o(f_1,\ldots,f_n)$.
\end{proof}
Now, we investigate the boundary structure of a polyhedral by means of convex geometry.
\begin{lemma}\label{exposed}
Let $F$ be an exposed proper face of $C(f_1,\ldots,f_n)$. Then there exists $u\in V\setminus\{0\}$ such that
\[
F=\{x\in C(f_1,\ldots,f_n):\, \mathrm{max}_{y\in C(f_1,\ldots,f_n )} \langle y, u \rangle=\langle x, u \rangle=0\}.
\]
Moreover, $F$ is itself a polyhedral.
\end{lemma}
\begin{proof}
Let $u\in V$ such that $F=\{x\in C(f_1,\ldots,f_n):\, \mathrm{max}_{y\in C(f_1,\ldots,f_n )} \langle y, u \rangle=\langle x, u \rangle=c\}$.  Since $C(f_1,\ldots,f_n)$ is a cone, if $x\in C(f_1,\ldots,f_n)$, then $tx \in C(f_1,\ldots,f_n)$ for any positive $t$. Hence the maximum $c$ must be zero and  so $F$ is a closed cone. We claim that there exists, $J=\{j_1,\ldots,j_s\}\subset \{1,\ldots,n\}$ such that $\langle f_{i_j}, u \rangle =0$ per $j=1,\ldots,s$ and $\langle f_r,u \rangle <0$ for any $r \in \{1,\ldots,n\} \setminus \{j_1,\ldots,j_s\}$. Otherwise there exist $\alpha, \beta \in \{1,\ldots, n \}$ such that $\langle f_\alpha, u \rangle \langle f_\beta , u \rangle <0$ and so $\mathrm{relint} [f_\alpha,f_\beta ] \cap F \neq \emptyset$ which is a contradiction. Now, it is easy to check that
$
F=C(f_{i_1},\cdots,f_{i_s})
$
and so the result is proved.
\end{proof}
Any face of a polytope is exposed \cite{schneider-convex-bodies}. The following statement proves that any face of a polyhedral is exposed as well.
A proof is given in \cite{schr} $p.101$ (see also  Proposition $1.22$ $p.6$ in \cite{villa}).
\begin{prop}\label{face-exposed}
Let $F$ be an exposed face of $C(f_1,\ldots,f_n)$ and $F_1 \subset F$ be an exposed face of $F$. Then $F_1$ is an exposed face of $C(f_1,\ldots,f_n)$. Hence any face of a  polyhedral is exposed.
\end{prop}
\begin{remark}
Let $E$ be a closed convex set of $V$. Let $F\subset E$ be face. If $F_1 \subset F$ is an exposed face of $F$, it is not true in general that $F_1$ is an exposed face of $E$.
This means that the above result holds for a polyhedral but it in not true in general.
\end{remark}
\section{Kempf-Ness functions}\label{sec:abstract-setting}
In this section we briefly recall the abstract setting introduced in \cite{bilio-zedda} (see also \cite{bgs,biliotti-raffero,bilio-ghigipr}).

Let $\spaz$ be a Hausdorff topological space and let $G$ be a connected real reductive group which acts continuously on $\spaz$. Observe that with these assumptions we can write $G=K\exp (\liep)$, where $K$ is a maximal compact subgroup of $G$. Starting with these data we  consider a
  function $ \Psi : \spaz \times G \ra \R$, subject to four conditions.
  \begin{enumerate}
  \item[($P1$)] For any $x\in \spaz$ the function $ \Psi(x,\cd )$
    is smooth on $G$.
  \item[($P2$)] The function $\Psi(x, \cd )$ is left--invariant
    with respect to $K$, i.e.: $\Psi(x,kg) = \Psi(x,g)$.
  \item[($P3$)] For any $x\in \spaz$, and any $\x \in \liep$ and
    $t\in \R${:}
    \begin{gather*}
      \frac{\mathrm{d^2}}{\mathrm{dt}^2 } \Psi(x,\exp(t\x)) \geq 0.
    \end{gather*}
    Moreover:
    \begin{gather*}
      \frac{\mathrm{d^2}}{\mathrm{dt}^2 }\bigg \vert_{t=0}
      \Psi(x,\exp(t\x)) = 0
    \end{gather*}
    if and only if $\exp(\R \x) \subset G_x$.
  \item[($P4$)] For any $x\in \spaz$, and any $g, h\in G$:
    \begin{gather*}
      \Psi(x,g) + \Psi({gx}, h) = \Psi(x,hg).
    \end{gather*}
    This equation is called the \emph{cocycle condition}.
  \end{enumerate}
For $x\in \spaz$, we define $\fun (x) \in \liep^*$ by requiring that:
\[
\fun(x)(\xi):=\frac{\mathrm{d}}{\mathrm{dt} }\bigg \vert_{t=0}
      \Psi(x,\exp(t\xi)).
\]
We call $\fun: \spaz \lra \liep^*$ the \emph{gradient map} of $(\spaz, G, K, \Psi)$. As immediate consequence of the definition of $\fun$ we have the following result.
   \begin{prop}\label{equivarianza}
  The map $\fun : \spaz \ra \liep^*$ is $K$-equivariant.
\end{prop}
\begin{proof}
  It is an easy application of the cocycle condition and the
  left-invariance with respect to $K$ of $\Psi(x,\cdot)$. Indeed,
  \begin{gather*}
    \begin{split}
      \fun (kx) (\xi) &=\deze \Psi(kx,\exp(t \xi)) \\
      &=\deze \Psi (x, \exp(t\xi)k)=\deze \Psi
      (x,k^{-1} \exp(t\xi) k)\\
      &=\deze \Psi\left (x,\exp(t\mathrm{Ad}(k^{-1})(\xi))\right
      )=\mathrm{Ad}^* (k) ( \fun (x))(\xi) .
    \end{split}
  \end{gather*}
\end{proof}
The following definition summarizes the above discussion.
\begin{defin}
\label{def-kn}
Let $G$ be a real reductive Lie group, $K$ a maximal compact subgroup of $G$ and $\spaz$ a
  topological space with a continuous $G$--action. A \emph{Kempf-Ness
    function} for $(\spaz, G,K)$ is a function
  \begin{gather*}
    \Psi : \spaz \times G \ra \R ,
  \end{gather*}
  that satisfies conditions ($P1$)--($P4$).
\end{defin}
 Let $(\spaz,G, K) $ be as above and let $\Psi$ be a Kempf-Ness
 function.
 \begin{defin}
   \label{stabilita}
   Let $x\in \spaz$. Then:
   \begin{enumerate}
   \item $x$ is \enf{polystable} if $G\cdot  x \cap \mume \neq \vacuo$.
   \item $x$ is \enf{stable} if it is polystable and $\lieg_x$ is
     conjugate to a subalgebra of $\liek$.
   \item $x$ is \enf{semi--stable} if
     $\overline{G \cdot x} \cap \mume \neq \vacuo$.
   \item $x$ is \enf{unstable} if it is not semi--stable.
   \end{enumerate}
 \end{defin}
\begin{remark}
   The four conditions above are $G$-invariant in the sense that if a
   point $x$ satisfies one of them, then every point in the orbit of
   $x$ satisfy the same condition. This follows directly from the definition
   for polystability, semi--stability and unstability, while for 
   stability 
    it is enough to recall that
   $\lieg_{g x} = \Ad(g) (\lieg_x)$.
 \end{remark}

The following result establishes a relation between the Kempf-Ness function and polystable points. A proof is given in \cite[p.2190]{bilio-zedda} (see also \cite{bgs,teleman-symplectic-stability,mundet-Trans}).
\begin{prop}\label{critical-point}
  Let $x\in \misu$. The following conditions are equivalent:
  \begin{enumerate}
  \item $g\in G$ is a critical point of $\Psi(x, \cd)$;
  \item $\fun(gx) =0$;
  \item $g\meno K$ is a critical point of $\psi_x$.
  \end{enumerate}
\end{prop}

\begin{prop}\label{compatible}
If $\fun (x)=0$, then the stabilizer of $x$, i.e., $G_x=\{g\in G:\, gx=x\}$  is compatible with respect to the Cartan decomposition of $G=K\exp(\liep)$. Moreover, if $G=\exp(\liep)$, with $\liep$ Abelian, then any stabilizer is compatible.
\end{prop}
\begin{proof}
The first part of the statement is well-known. A proof is given in    \cite{bilio-zedda}, see also \cite{heinzner-schwarz-Cartan}. The second part is also easy to check. For the sake of completeness, we give a proof.

Let
$x\in \spaz$ and let $g\in G_x$. Then $g=\exp (\x)$ for some $\x \in \liep$. Let $f(t):=\fun^{\x} (\exp (t\x)x)=\fun (\exp ( t \x ) x) (\x )$. Then
   $f(0)=f(1)=0$ and  \begin{gather*}
     \desudt f(t)=\desudt \fun^{\x} (\exp
     (t\x)x)=\frac{\mathrm{d^2}}{\mathrm{dt}^2 } \Psi(x,\exp(t\x))
     \geq 0.
   \end{gather*}
   Therefore
   $\frac{\mathrm{d^2}}{\mathrm{dt}^2 } \Psi(x,\exp(t\x))=0$ for
   $0\leq t \leq 1$.  It follows from ($P3$) that $\exp(t \x)x=x$ for any $t\in \R$ and thus $\exp(t\xi ) \in G_x$.
\end{proof}

Given $\xi \in \liep$ for any $t\in \R$ we define
$
\lambda(x,\xi,t)=\fun (\exp(t\xi )x )(\xi).
$
Applying the cocycle condition we get
\[
\fun(\exp
     (t\x)x) (\xi)=\desudt \Psi(x,\exp(t\xi)
\]
and so
\[
\desudt \fun(\exp
     (t\x)x) (\xi) =\frac{\mathrm{d^2}}{\mathrm{dt}^2 } \Psi(x,\exp(t\x))
     \geq 0.
\]
This means that
\[
\lambda (x,\xi,t)=\fun(x)(\xi)+\int_0^t \frac{\mathrm{d^2}}{\mathrm{ds}^2 } \Psi(x,\exp(s\x)) \mathrm{d}s
\]
is a non decreasing function as a function of $t$. Moreover,
\[
\Psi(x,\exp(t\xi))=\int_0^t \lambda (x,\xi,\tau) \mathrm d \tau,
\]
and so
\[
\lambda(x,\xi):=\lim_{t\mapsto +\infty} \desudt \Psi(x,\exp (t\xi)) \in \R \cup \{\infty\}.
\]
The function $\lambda(x,\cdot)$ is called maximal weight of $x$ in the direction $\xi$. For a reference see, amongst many others,  \cite{bgs,bilio-zedda,mumford-GIT,mundet-Trans, teleman-symplectic-stability}. We point out that the maximal weight is well defined for any convex function.

Let $V$ be a finite dimensional real vector space and let $f:V \lra \R$ be a convex function. For any $\xi \in V$, the function $g(t)=f(t\xi)$ is convex and so
\[
\lambda_f (\xi)=\lim_{t\mapsto +\infty} \desudt f(t \xi) \in  \R \cup \{\infty\}
\]
is well-defined. We conclude this section proving a useful lemma.
\begin{lemma}\label{exaustion}
 Let  $f:V  \lra \R$ be a convex function. Assume that for any $\xi \in V \setminus\{0\}$, we have $\lambda_f (\xi) >0$. Then $f$ is an exhaustion and so it has a critical point which is a global minimum.
 \end{lemma}
\begin{proof}
We may assume that $V$ is endowed by a scalar product $\scalo$. Denote by $\mathrm{S}(V)$ the unit sphere with respect to $\scalo$. Let $\xi \in \mathrm{S}(V)$. Since $\lambda_f(\xi)>0$, keeping in mind that $\frac{\mathrm{d^2}}{\mathrm{dt}^2 } f(\exp(t\x)) \geq 0$, it follows there exist $t(\xi)>0$ and $C_o>0$ such that
\[
\desudt f(\exp(t(\xi) \xi) \geq C_o >0,
\]
for any $t\geq t(\xi)$.
Hence there exists a neighborhood $U_\xi$ of $\xi$ in $\mathrm{S}(V)$ such that $\desudt f(t \nu) > \frac{C_o}{2}$ for any $t\geq t(\xi)$ and for any $\nu \in U_\xi$.  By usual compactness argument,  there exist two constants $C>0$ and $t_o >0$ such that
$
\desudt f(t \xi) \geq C,
$
for any $\xi \in \mathrm{S}(V)$ and for any $t\geq t_o$. Therefore, for any $v\in V$ such that $\parallel v \parallel \geq t_o$, we get
\[
f(v)=f(t_o \frac{v}{\parallel v \parallel}) + \int_{t_o}^{\parallel v \parallel} \desudt f \left( t \frac{v}{\parallel v \parallel}\right) \mathrm d t
\]
and so  $f(v)\geq \mathrm{min}_{ \parallel w \parallel =t_o} f(w)$. This means that  $f$ is an exhaustion and so it  has a critical point which is a global minimum.
\end{proof}
\section{Convexity Theorems for Abelian groups}
Let $G$ be a connected real reductive Lie group and let $\rho:G \lra \mathrm{GL}(V)$ be a faithful representation on a finite dimensional real vector space $V$. We identify $G$ with $\rho(G)\subset \mathrm{GL}(V)$ and we assume that $G$ is closed and it is closed under transpose. This means that there exists a scalar product $\scalo$ on $V$ such that $G=K\exp(\liep)$, where $K\subset \mathrm{O}(V)$ and $\liep \subset \lieg \cap \mathrm{Sym}(V)$. In the sequel, we denote by
$\parallel \cdot \parallel=\sqrt{\scalo}$. We define
\[
\Psi:V \times G \lra \R \qquad \Psi(x,g)=\frac{1}{2}\left(\parallel g x \parallel^2 - \parallel x \parallel^2 \right).
\]
\begin{lemma}
$\Psi:V \times G \lra \R$ is a Kempf-Ness function and the corresponding gradient map $\fun_\liep : V \lra \liep^*$ is given by
$
\fun(x)(\xi)=\langle \xi x, x \rangle.
$
\end{lemma}
\begin{proof}
$(P1)$ and $(P2)$ are easy to check. Let $\xi \in \liep$ and let $f(t)=\Psi(x,\exp(t\xi))$. Then
\[
f'(t)=\langle \exp(t \xi ) \xi x, \exp(t\xi) x \rangle, \qquad f''(t)= \langle \exp(t \xi ) \xi x, \exp(t\xi) \xi x \rangle.
\]
Hence $f''(t) \geq 0$ and $f''(0)=0$ if and only if $\xi x=0$ and  $\exp(\R \xi ) \subset G_x$. Now,
\[
\begin{split}
\Psi(x,hg)&=\frac{1}{2} \left(\parallel hg x \parallel^2 - \parallel x \parallel^2 \right) \\
          &=\frac{1}{2} \left(\parallel hg x \parallel^2 - \parallel gx \parallel^2 \right) + \frac{1}{2} \left(\parallel g x \parallel^2 - \parallel x \parallel^2 \right) \\
          &= \Psi(gx,h)+\Psi(x,g),
\end{split}
\]
proving the cocycle condition. Finally
\[
\fun (x)(\xi)=\frac{1}{2} \desudtzero \langle \exp (t\xi) x ,\exp(t\xi) x \rangle =\langle \xi x , x \rangle,
\]
concluding the proof.
\end{proof}
Let $A=\exp(\lia)$, where $\lia\subset \liep$ is an Abelian subalgebra. It is easy to check that $\Psi_{|_{A\times V}}$ is a Kempf-Ness function with respect to $A$ and $\fun_\lia :=(\fun_\liep)_{|_{\lia}}$ is the corresponding gradient map \cite{bilio-zedda}.
Since $\lia$ is an Abelian subalgebra of symmetric endomorphisms, they are simultaneously diagonalizable. Hence there exists an orthonormal basis $\{v_1,\ldots,v_n\}$ of $V$ and functionals $\alpha_1\ldots,\alpha_n\in \lia^*$ such that
\[
\xi (v)=\sum_{i=1}^n \alpha_i (\xi )\langle v,v_i \rangle v_i.
\]
This means that if  $v=x_1 v_1 +\cdots +x_n v_n$, then
\[
\exp(\xi)(v)=e^{\alpha_1 (\xi)}x_1 v_1 + \cdots +e^{\alpha_n (\xi)} x_n v_n.
\]
In particular,
\[
\fun_\lia (v)=\sum_{i=1}^n \parallel x_i \parallel^2 \alpha_i,
\]
and so the image of $\fun_\lia$  is contained in the polyhedral $C(\alpha_1,\ldots,\alpha_n)$.

Let $x=x_1 v_1 + \cdots + x_n v_n$.  We define $\mathrm{supp}_x=\{i\in \{1,\ldots,n\}:\, x_i \neq 0\}$. If $I\subset \{1\,\ldots,n\}$, we denote by $C_I$  the polyhedral generated by $\{\alpha_i:\, i\in I\}$, i.e.,
\[
C_I=\left\{ \sum_{i\in I} s_i \alpha_i:\, i\in I\, \mathrm{and}\ s_i\geq 0 \right\}.
\]
 We also denote by $C_I^o=\left\{\sum_{i\in I} s_i \alpha_i:\, i\in I\ \mathrm{and}\ s_i > 0 \right\}$. In the theory of real reductive representations,  a fundamental problem  is to compute the image of the gradient map. The following theorem generalizes a result proved by Kac-Peterson \cite{kacp}, see also \cite{vergne}, to the real case.
\begin{teo}\label{orbita}
Let $x\in V$ and let $I=\mathrm{supp}_x$. Then the map $\fun_\lia:\overline{A\cdot x} \lra \lia^*$ satisfies:
\begin{enumerate}
\item $\fun_{\lia}: A\cdot x \lra C^o_I$ is a diffeomorphism onto. Hence $A/A_x \cong C_I^o$;
\item $\fun_\lia :\overline{A\cdot x } \lra C_I$ is an homeomorphism and for any face $\sigma \subset C_I$  there exists a unique $A$-orbit $Y$ such that $\fun_\lia^{-1}(\sigma)=\overline{Y}$. Therefore $\fun_\lia (\overline{A\cdot x})$ is a polyhedral.
\end{enumerate}
\end{teo}
\begin{proof}
By  Proposition $2.1$ in \cite{bilio-ghigipr}, the map  $\fun_\lia :A\cdot x \lra \lia^*$ is a diffeomorphism and its image is a convex open subset of $\mu(x)+\lia_x^o$, where $\lia_x^o=\{\phi\in \lia^*:\, \phi_{|_{\lia_x}}=0\}$. It is easy to check that $\fun_\lia (A\cdot x)$ is invariant under multiplication of non negative real numbers and so it is an open convex cone contained in $C^o_I$. Now, we shall prove that $\fun_\lia (A\cdot x)=C_I^o$.

Let $\mathfrak b\subset \lia$ such that $\lia=\lia_x \oplus \mathfrak b$. Let $c\in C^o_I$. Then $c=\sum_{i\in I} c_i \alpha_i$ with $c_i>0$. We define
\[
f:\mathfrak b \lra \R \qquad f(\xi)=\Psi(x,\exp(\xi))-  c (\xi).
\]
The equation $\fun_\lia  (\exp(\xi_o)x)=c$ means the existence of a critical point of $f$.
We prove that $f$ is strictly  convex  and an exhaustion.

Fix $v ,w\in \mathfrak b$ with
  $w \neq 0$ and consider the curve $\gamma(t)=v+tw$.  Set
  $u(t)=\fun_\lia (\gamma(t))$. It is easy to check that  $u''(t)=\frac{\mathrm{d^2}}{\mathrm{dt}^2 }\bigg \vert_{t=0}
      \Psi(\exp(v)x,\exp(t w)) >   0$ since $w \notin \lia_x$.
This proves  $f$ is a
strictly convex function on $\mathfrak b$.

Let $\xi \in \mathfrak b \setminus\{0\}$. Then
\[
\frac{d}{dt} f(t \xi)=\sum_{i\in I} \langle( e^{t \alpha_i (\xi)} \parallel x_i \parallel^2 -c_i \rangle)  \alpha_i (\xi) .
\]
Since $\xi \notin \mathfrak{a}_x$, it follows $ \alpha_i ( \xi ) \neq 0$ for some $i\in I$. If $\alpha_i (\xi) >0$, then
\[
\lim_{t\mapsto +\infty} (e^{t \alpha_i (\xi)} \parallel x_i \parallel^2 -c_i ) \alpha_i ( \xi)  =+\infty.
\]
If $\alpha_i ( \xi ) <0$, then $\lim_{t\mapsto +\infty} e^{t \alpha_i (\xi)} \parallel x_i \parallel^2 =0$ and so, keeping in mind that $c_i>0$, we get
\[
\lim_{t\mapsto +\infty} (e^{t \alpha_i (\xi)}\parallel x_i \parallel^2 -c_i ) \alpha_i ( \xi )=-c_i \alpha_i (\xi )>0.
\]
Hence $\lambda_f (\xi)>0$ for every $\xi \in \mathfrak b\setminus\{0\}$. By Lemma \ref{exaustion} the function $f$ is an exhaustion and so  it has a critical point concluding the proof of item $(a)$.

Let $\sigma$ be a face of $C_I$. By Proposition \ref{face-exposed}, there exists $\xi \in \lia$ such that $\sigma=F_\xi (C_I )$. By Lemma \ref{exposed}  there exists $J\subset I$ such that $\alpha_i (\xi) =0$ for $i\in J$ and $\alpha_i (\xi)<0$ otherwise. In particular
\[
\lim_{t\mapsto +\infty} \exp (t\xi) x=\sum_{i\in J} x_i v_i=\theta
\]
and $\fun (\theta)\in \sigma$. We prove
\[
\fun_\lia^{-1}(\sigma)=\{v\in \overline{A\cdot x}:\, \mathrm{max}_{z\in \overline{A\cdot x} } \fun_\lia (v)(\xi) =\fun_\lia (z)(\xi)=0\}=\overline{A\cdot \theta}
\]
Let
$u\in \fun_\lia^{-1}(\sigma)$. Write $s(t)=\fun_\lia (\exp(t\xi )u)(\xi)$. The function $s$ has a maximum in $t=0$ and so
\[
\dot{s}(0)=2 \langle \xi u, \xi u \rangle=0.
\]
This implies $\fun_\lia^{-1}(\sigma)\subset \mathrm{Ker}\, \xi$. Since $\xi$ commutes with $A$ it follows  $\mathrm{Ker}\, \xi$ is $A$-invariant. Moreover, the gradient flow of the function $\fun_\lia^{\xi}(x):=\fun_\lia(x)(\xi)$ is given by $\exp(t\xi)$. Therefore $\mathrm{Crit}\, \fun_\lia^\xi=\mathrm{Ker}\, \xi$ and so $\fun_\lia^{-1}(\sigma)$ is $A$-invariant as well. This implies
$
\overline{A\cdot \theta} \subset \fun_\lia^{-1}(\sigma).
$

Let $z\in \fun_\lia^{-1}(\sigma)$. Since $z\in \overline{A\cdot x}$, there exists $\{a_n\}_{n\in \mathbb N} \in A$ such that $a_n x \mapsto z$. The flow $\exp(t\xi)$ commutes with $A$ and so for any $n \in \mathbb N$, we have
\[
\lim_{t\mapsto +\infty} \exp(t\xi) (a_n x)= a_n \theta \in \mathrm{Ker}\, \xi.
\]
This means $\lim_{t\mapsto +\infty} \exp(t\xi) a_n x=P(a_n x)=a_n P(x)=a_n \theta$, where $P:V \lra \mathrm{Ker}\, \xi$ is the orthogonal projection  on $\mathrm{Ker}\, \xi$. Since $a_n x \mapsto z$ and $z\in \mathrm{Ker}\, \xi$, it follows  $P(a_n x )=a_n \theta \mapsto z$. This implies $z\in \overline{A\cdot \theta}$ and so $\fun_\lia^{-1}(\sigma)=\overline{A\cdot \theta}$.  Now, applying again item $(a)$, we get $\fun_\lia:A\cdot \theta \lra C_{I'}$,  is a diffeomorphism, where $I'=\mathrm{supp}_{\theta}$. In particular $\theta \in \mathrm{relint} \sigma$. Summing up we have proved that the map $\fun_\lia :\overline{A\cdot x} \lra C_I$ is an  homeomorphism.
\end{proof}
\begin{cor}\label{cor}
Let $x\in V$ and let $I=\mathrm{supp}_x$. Let $F$ be a face of $C_I$ and let $J\subset I$ be the subset of $I$ associated to $F$ as in the proof of Theorem \ref{orbita}. Let $v_F =\sum_{i\in J} x_i v_i$. Then
\[
\fun_\lia : A \cdot v_F \lra \mathrm{relint}\, F,
\]
is a diffeomorphism onto. Moreover, $\overline{A\cdot x}$ is the disjoint  union  of orbits $A\cdot v_F$, as  $F$ runs over the set of faces of the polyhedral $C_I$.
\end{cor}
\begin{proof}
Let $v_F$ the element associated to $F$. In Theorem \ref{orbita} we have proved that $\fun_\lia (v_F) \in \mathrm{relint} F$, $\fun_\lia : A \cdot v_F \lra \mathrm{relint} F$ is a diffemorphism and $\fun_\lia^{-1}(F)=\overline{A\cdot v_F}$. Therefore  $\overline{A\cdot x}$ is the disjoint  union  of orbits $A\cdot v_F$, as  $F$ runs over the set of faces of the polyhedral $C_I$.
\end{proof}
\begin{teo}[Hilbert-Mumford criterion]\label{hm}
Let $u\in \overline{A\cdot x}$. Then there exist $\xi \in \lia$ and $a\in A$ such that
\[
\lim_{t\mapsto +\infty} \exp(t\xi ) a x=u.
\]
\end{teo}
\begin{proof}
By Corollary \ref{cor} there exists a unique face $F$ such that $\fun (u)\in \mathrm{relint}\,  F$. Therefore $u=a v_F$. Taking $\xi\in \lia$ such that $F=F_\xi (C_I)$ it follows
\[
\lim_{t\mapsto +\infty} \exp(t\xi) ax=av_F=u.
\]
\end{proof}
\begin{cor}
$A\cdot x$ is closed if and only if $0\in C^o_I$. Therefore $A\cdot x$ is closed if and only if $A\cdot x \cap \fun_\lia^{-1} (0) \neq \emptyset$.
\end{cor}
\begin{proof}
By Theorem \ref{orbita}, $A\cdot x$ is closed if and only if $C_I^o$ is closed. By Lemma \ref{cono-chiuso} we get $A\cdot x$ is closed if and only if $0\in C_I^o$ and so if and only if $A\cdot x \cap \fun_\lia^{-1} (0) \neq \emptyset$.
\end{proof}
\begin{teo}\label{null-cone-abelian}
The set $\{x\in V:\, 0 \in \overline{A\cdot x}\}$ is a real algebraic subset of $V$ and so it is closed.
\end{teo}
\begin{proof}
By Theorem \ref{hm}, $0\in \overline{A\cdot x}$ if and only if $0$ is a face of $C_I$, where $I=\mathrm{supp}_v$. Since there exist a finite numbers of  $C_I$ where $I\subset \{1,\ldots,n\}$, it follows that there exist $I_1,\cdots,I_k \subset \{1,\ldots,1\}$ such that $0\in \overline{A\cdot x}$ if and only if $\mathrm{supp}_x=I_j$ for some $j\in \{1,\ldots,k\}$. Since any face of $C_I$ is exposed,  there exist $\xi_1,\ldots,\xi_k \in \lia$ such that $0\in \overline{A\cdot x}$ if and only if $\exp(t\xi_s ) x \mapsto 0$ for some $s\in \{1,\ldots,k\}$. Now, $\exp (t \xi_s ) \cdot x\mapsto 0$ if and only if $x=\sum_{i\in J} x_i v_i$ with $\alpha_i (\xi_s)< 0$ for any $i\in J$.

Let $Z_s=\{i\in \{1,\ldots,n\}:\, \alpha_i (\xi_s ) < 0 \}$, for $s=1,\ldots,k$. Define $\mathcal H_s =\{x\in V:\, \langle x, v_k\rangle=0$ for $k\in \{1,\ldots,n\} \setminus Z_s\}$. It is easy to check $\exp(t\xi_s ) x \mapsto 0$ if and only if $x\in \mathcal H_s$. Therefore
\[
\{x\in V:\, 0 \in  \overline{A\cdot x}\}=\mathcal H_1 \cup \cdots \cup \mathcal H_k,
\]
concluding the proof.
\end{proof}
\begin{prop}
The image of $\fun_\lia :V \lra \lia^*$ is a polyhedral and the set $\{x\in V:\, \fun_\lia (\overline{A\cdot x})=\fun_\lia (V)\}$ is an open and dense subset of $V$.
\end{prop}
\begin{proof}
Let $x=x_1 v_1+ \cdots +x_n v_n$ such that $x_i \neq 0$ for any $i=1,\ldots,n$. Since $\mathrm{supp}_x=I=\{1,\ldots,n\}$, by  Theorem \ref{orbita} we get $\fun_\lia (\overline{A\cdot x})=C_I$. On the other hand, by definition of $\fun_\lia$, it follows $\fun_\lia (V)\subset C_I$ and so the image of $\fun_\lia$ is the polyhedral $C_I$.

Let $x\in V$ and let $I=\mathrm{supp}_x$. By Theorem \ref{orbita}, there exists a neighborhood $U$ of $x$ such that for any $y\in U$ we have $C_I^o \subset \fun_\lia (\overline{A\cdot y})$ for any $y\in U$. Therefore the set $\{x\in V:\, \fun_\lia (\overline{A\cdot x})=\fun_\lia (V)\}$ is open and it contains an open dense subset of $V$. Therefore it is an open dense subset of $V$.
\end{proof}
Now we investigate convexity Theorems for $A$-invariant subsets. The proof of next result uses original ideas from \cite{vergne}.
\begin{prop}\label{convexity-irruducible}
Let $M$ be a closed real algebraic irreducible subset of $V$. Assume that $M$ is $A$-invariant. Then
$\fun_\lia (M)=C_I$ where $I=\supp_{v}$ for some $v\in M$ and so the image is a polyhedral.
\end{prop}
\begin{proof}
Let $v\in M$ and let  $I=\mathrm{supp}_v$.  Define $U_v:=\{u\in M:\, C_{I}\subset C_{\mathrm{supp}_u}\}$. The set $\{u\in V:\, C_I \subset C_{\mathrm{supp}_u}\}$ is Zariski open and $v\in U_v$. Therefore $U_v$ is Zariski open.  Now, keeping in mind there exist finitely many subsets of $\{1,\ldots,n\}$, there exist a finite numbers of open subset $U_v$ and so
\[
M=U_{v_1} \cup \cdots \cup U_{v_k},
\]
for some $v_1,\ldots,v_k \in M$.
Since $M$ irreducible, it follows that $U_{v_1} \cap \cdots \cap U_{v_k} \neq \emptyset$, and so it is Zariski open. Therefore there exists $x\in  U_{v_1} \cap \cdots \cap U_{v_k}$ and so $\fun_\lia (M)=\fun_\lia (\overline{A\cdot x})=C_I$ where $I=\mathrm{supp}_x$.
\end{proof}
We conclude this section computing the image of $G$ orbits under the gradient map.
\begin{teo}\label{convex-orbit-general}
Let $x\in V$. There exists $v\in G\cdot x$ such that $\fun_{\lia}(\overline{G\cdot x})=\fun_{\lia}(\overline{A\cdot v})$ and so the image is a polyhedral.
\end{teo}
\begin{proof}
We give two proofs. In the first proof we assume that $G$ is an algebraic real reductive group and $\rho:G \lra \mathrm{GL}(V)$ is a rational representation. Let $G^\C$ be the complexification of $G$ acting on $V^\C$. Let $G_\R =G^{\C} \cap \mathrm{GL}(V)$. Then
$\overline{G^\C \cdot x \cap V}$ is a closed real algebraic irreducible set and  $G^\C \cdot x \cap V$ is a finite union of $G_\R$ orbits \cite{bgs,berlein-jablonki,whitney}. Moreover,  any $G_\R^o $ orbit throughout an element $v\in  G^\C \cdot x \cap V$ is open \cite[Proposition 2.3]{bgs}. Since $(G_\R)^o=G$ \cite{borel-book,chevally} it follows that any $G$ orbit throughout an element of  $G^\C \cdot x \cap V$ is also open.
Write $M=\overline{G^\C \cdot x \cap V}$. By Proposition \ref{convexity-irruducible} there exist  $v_1,\ldots,v_k \in M$ such that $U_{v_1} \cap \cdots \cap U_{v_k} \neq \emptyset$ and $M=U_{v_1} \cup \cdots \cup U_{v_k}$. Since $U_{v_1} \cap \cdots \cap U_{v_k}$ is Zarisky open,  it follows
$G\cdot x \cap U_{v_1} \cap \cdots \cap U_{v_k} \neq \emptyset$, and so $\fun_\lia (\overline{G\cdot x})=\fun_\lia (\overline{A \cdot z})$ for some $z\in G\cdot x$.

The second proof works without any algebraic assumption. The main tool is a Theorem of Baire.

Let $y=y_1 v_1+\cdots + y_n v_n$ with $y_1\cdots y_n \neq 0$. If $y\in \overline{G\cdot x}$, then
$\fun_\lia (\overline{A\cdot y})=\fun_\lia (V)$ and so
the result is proved. Otherwise,
\[
\overline{G\cdot x} = \overline{G\cdot x} \cap \{v\in V:\, \langle v, v_1 \rangle =0\}\cup  \cdots \cup \overline{G\cdot x} \cap \{v\in V:\, \langle v, v_n \rangle =0\}.
\]
By a Theorem of Baire, there exists $k\in \{1,\ldots,n \}$ such that $\overline{G\cdot x} \cap \{v\in V:\, \langle v, v_k \rangle =0\}$ has an interior point. Therefore, there exists $y\in G\cdot x$ and a neighborhood $V$ of the origin of the Lie algebra $\lieg$ of $G$ such that $\exp(\theta)y\in \{v\in V:\, \langle v, v_k \rangle =0\}$ for any $\theta \in V$. Assume by contradiction that $G\cdot y$ is not contained in $V_k=\{v\in V:\, \langle v, v_k \rangle =0\}$. Write
$
A=\{z \in G\cdot y:   z \in V_k  \}.$  Since $G$ is analytic and the $G$ action on $V$ is analytic,  $G\cdot y$ is analytic.
By the above discussion the interior of $A$, that we denote by $A^o$, is not empty. Let $z\in \partial A^o$. Let $\phi:U \lra U'$ be a chart of $z$. Then $z\in V_k$ and so  $\phi^{-1}(U'\cap V_k)$ contains an open subset. Since $\phi$ is analytic it follows $\phi (U)\subset V_k$. A contradiction. Therefore $\overline{G\cdot x}=\overline{G\cdot y}\subset V_k$. In particular, there exists a $G$-invariant subspace $Z\subset V_k$ containing $\overline{G\cdot x}$ such that $V=Z\oplus Z^\perp$ is a $G$-stable splitting. This follows by the Cartan decomposition $G=K\exp(\liep)$ where $K\subset \mathrm{O}(V)$ and $\liep \subset \mathrm{Sym}(V)$. Moreover, if $x\in Z$, then
\[
\fun_\lia (\overline{G\cdot x})=(\fun_{\lia})_{|_{Z}}(\overline{G \cdot x}).
\]
After a finite number of steps,  there exists a $G$ invariant subspace $W$ such that $\overline{G\cdot x}\subset W$ and $\overline{G\cdot x}$ is not contained in any subspace of $W$, i.e., it is full.
Hence, there exists $y\in \overline{G\cdot x}$ such that $\fun_\lia (\overline{G\cdot x})=\fun_\lia (\overline{A\cdot y})=\fun_\lia (W)$. Moreover, since
$
\{z\in W :\, \fun_\lia (\overline{A\cdot z})=\fun_\lia (W) \}
$
is open and dense, we may choose $y\in G\cdot x$ such that $\fun_\lia (\overline{G\cdot x})=\fun_\lia (\overline{A\cdot y})=\fun_\lia (W)$.
\end{proof}
\section{Convexity Theorems for real reductive representations}\label{section-reductive}
Let $(V,\scalo)$ be a real finite dimensional vector space and let $G\subset \mathrm{GL}(V)$ be a connected closed real reductive group such that $G=K \exp(\liep)$, where $K\subset \mathrm{O}(V)$ and $\liep\subset \lieg \cap \mathrm{Sym}(V)$. Let $\fun_\liep:V \lra \liep^*$ be the corresponding gradient map. Using an $\mathrm{Ad}(K)$-invariant scalar product on $\liep$ we can identify $\liep$ with $\liep^*$ and so we may think the gradient map as a $\liep$-valued map. If $\lia\subset \liep$ is a an Abelian subalgebra, then $\fun_\lia =\pi_\lia \circ \fun_\liep$ is the gradient map associated to $A=\exp(\lia)$, where $\pi_\lia: \liep \lra \lia$ is the orthogonal projection.
In this section we study the convex hull of the image of the closure of $G$ orbits under the gradient map associated to $G$.
\begin{teo}\label{convo}
Let $x\in V$ and let $E=\mathrm{Conv} (\fun_\liep (\overline{G\cdot x}))$. Then $E$ is a closed convex set and any face of $E$ is exposed.
\end{teo}
\begin{proof}
By Theorem \ref{convex-orbit-general},  $C=\fun_\lia (\overline{G\cdot x})$ is a polyhedral. Since $\fun_\lia = \pi_\lia \circ \fun_\liep$ and $\fun_\lia (\overline{G\cdot x})$ is convex, it follows $\pi_\lia (E)=\pi_\lia (\fun_\liep (\overline{G\cdot x}))=\fun_\lia (\overline{G\cdot x})$. By Lemma $7$ in \cite{gichev-polar} we have  $E=K C$. In particular $E$ is closed, due to the fact that $K$ is compact, and $C$ is closed. By Theorem \ref{exposed} any face of $C$ is exposed and so, by Lemma $1.4$ and Proposition $1.4$ in \cite{bgh-israel-p}, any face of $E$ is exposed.
\end{proof}
\begin{remark}
In \cite{bgh-israel-p} the authors study $K$-invariant compact  convex sets. However, Lemma $1.4$ and Proposition $1.4$ holds without the compactness assumption.
\end{remark}
Similarly, one may prove the following more general result.
\begin{teo}\label{convo}
Let $M$ be an $G$-invariant closed algebraic irreducible subset of $V$. Let $\lia\subset \liep$ a maximal Abelian subalgebra. Then
$
\mathrm{Conv} (\fun_\liep (M))=KC,
$
where $C$ is a polyhedral. Therefore $\mathrm{Conv} (\fun_\liep (M))$ is closed and any face of $\mathrm{Conv} (\fun_\liep (M))$ is  exposed.
\end{teo}
Given $\beta \in \liep$, we define the following subgroups:
\begin{gather*}
G^\beta=\{g\in G:\, \mathrm{Ad}(g)(\beta)=\beta\}\\
  G^{\beta-} :=\{g \in G : \lim_{t\mapsto + \infty} \exp({t\beta}) g
  \exp({-t\beta}) \mathrm { exists} \}\\
  R^{\beta-} :=\{g \in G : \lim_{t\mapsto + \infty} \exp({t\beta})
  g \exp({-t\beta}) =e \}
\end{gather*}
The next lemma is well-known. A proof is given in \cite{biliotti-ghigi-heinzner-2}.
\begin{lemma}
If $g\in G^{\beta-}$ then $ \lim_{t\mapsto + \infty} \exp({t\beta}) g
  \exp({-t\beta})\in G^\beta$. Moreover,
  $G^{\beta-} $ is a parabolic subgroup of $G$ with Lie algebra
  $\mathfrak g^{\beta-}=\lieg^\beta\oplus \mathfrak r^{\beta-}$ and $G=G^{\beta-}K$.  Every parabolic subgroup of $G$ equals
  $G^{\beta-}$ for some $\beta \in \liep$.  $R^{\beta-}$ is the
  unipotent radical of $G^{\beta-}$ and $G^\beta$ is a Levi factor. Finally, $G=KG^{\beta-}$.
\end{lemma}
We establish a connection between the Hilbert-Mumford criterion and the convex hull of the image of the gradient map restricted on the closure of a $G$ orbit.
\begin{prop}
Let $x\in V$. If $0\in \liep$ belongs to a face of $\mathrm{Conv}(\fun_\liep(\overline{G\cdot x}))$ then there exists $\xi \in \liep$ such that $\lim_{t\mapsto +\infty} \exp (t\xi) x=0$.
\end{prop}
\begin{proof}
Assume $0\in \mathrm{Conv}(\fun_\liep(\overline{G\cdot x}))$ belongs to a face. By Theorems \ref{convo} and \ref{convex-orbit-general} there exists $z\in G\cdot x$ such that
$
\mathrm{Conv} (\fun_\liep (\overline{G\cdot x}))=K \fun_\lia (\overline{A\cdot z}).
$
By Lemma $1.4$ \cite{bgh-israel-p} $p.430$, $0$ belongs to a face of $\overline{A\cdot z}$. Applying the Hilbert Mumford criterion for Abelian groups, there exists $\nu \in \liep$ such that
\[
\lim_{t\mapsto +\infty} \exp(t\nu ) z=0.
\]
and so $0\in \overline{G\cdot x}$. Now, $z=gx$ and $g=hk$ where  $h\in G^{\nu-}$ and $k\in K$. Therefore, keeping in mind that the limit  $\lim_{t\mapsto+\infty } \exp(t\nu ) h \exp(-t \nu )$ exists, we have
\[
0=\lim_{t\mapsto +\infty} \exp(t\nu) hkx =\lim_{t\mapsto +\infty} \left(\exp(t\nu) h \exp(-t\nu) \right) \exp(t\nu) kx=\lim_{t\mapsto +\infty} \exp(t\nu) kx
\]
Since $\exp(t\nu) kx=k\exp(t\mathrm{Ad}(k^{-1}) (\nu) x$ it follows
\[
\lim_{t\mapsto +\infty} \exp ( \mathrm{Ad}(k^{-1}) (\nu)) x=0
\]
and so the result is proved.
\end{proof}
A conjecture might be that the vice-versa holds as well. This means that $0\in \overline{G\cdot x}$ if and only if $0$ belongs to a face of $\mathrm{Conv}(\fun_\liep (\overline{G\cdot x}))$. Roughly speaking the Hilbert-Mumford criterion for real reductive groups follows from the boundary structure, by means of the convex geometry,  of the convex hull of the image of the gradient map restricted on the closure of $G$ orbits. We briefly recall a proof of the Hilbert-Mumford criterion \cite{rs,fb,heinzner-schwarz-Cartan}.
The main point is the existence of $k\in K$ such that $0\in \overline{A\cdot k x}$ and so $0\in \overline{A' \cdot x}$ where $A'=k^{-1}Ak$. Now, if $\fun_{\lia'}(\overline{G\cdot x})=\fun_{\lia'}(\overline{A' \cdot x})$, where $\lia'$ is the Lie algebra of $A'$, then Theorem \ref{convex-orbit-general}  implies $\mathrm{Conv}(\fun_\liep (\overline{G\cdot x}))=K\fun_{\lia'} (\overline{A' \cdot x})$ and so, by Lemma $1.4$ \cite{bgh-israel-p}, $0$ is a face of $\mathrm{Conv} (\fun_{\liep} (\overline{G \cdot x}))$. However $\fun_{\lia'}(\overline{G\cdot x})$ does not coincide in general with $\fun_{\lia'}(\overline{A' \cdot x})$. Indeed,  $0$ does not belong in general to a face of $\mathrm{Conv}(\fun_\liep (\overline{G\cdot x}))$.
\begin{ese}
Let $\mathrm{SL}(2,\R)$ acting on $\R^2$ in a natural way. It has $2$ orbit: the origin and $\R^{2}\setminus\{0\}$. Let $\xi$ be any  non zero symmetric matrices with trace $0$. If we consider a orthonormal basis $\{e_1,e_2\}$ of eigenvectors of $\xi$, up to rescaling, we may assume $\xi= \left(\begin{array}{cc} 1 & 0 \\ 0 & -1 \end{array}\right)$. The gradient map with respect to $A=\exp (\R \xi)$ is given by
$
\fun_\lia (x,y)=x^2-y^2.
$
Then $\fun_\lia (\R^2)=\R$ and so $\fun_\liep (\R^2)=\liep$.
\end{ese}
Now, we come back to the proof of the Hilbert-Mumford criterion. We have proved that $0\in \overline{A'\cdot x}$. Applying Hilbert Mumford criterion for Abelian groups, there exists
$\xi \in \lia'$, where $\lia'$ is the Lie algebra of $A'$,  such that $\exp(t\xi) x \mapsto 0$. We claim that the orbit $G^{\xi-} \cdot x$ goes to $0$ under the flow $\exp(t\xi)$ when $t\mapsto +\infty$ . Indeed,
\[
\lim_{t\mapsto +\infty} \exp(t \xi) g x=\lim_{t\mapsto +\infty} \left( \exp (t\xi) g \exp(-t \xi ) \right) \exp (t \xi ) x=0.
\]
Therefore, the function $x \mapsto \fun_\liep (x)(\xi)$ restricted on $\overline{G^{\xi-} \cdot x}$  has a maximum achieved in $0$. Now, if we denote by $E=\mathrm{Conv}(\fun_{\liep} (\overline{G^{\xi-}\cdot x}))$, the above discussion proves that $0$ belongs to an exposed  face of $E$. Roughly speaking, it seems that we can provide some appropriate connection between the Hilbert-Mumford criterion and the convex hull of the image of the gradient map restricted on the closure of orbits of parabolic subgroups. We leave this
problem for future investigation.
\section{Real reductive representations on Projective spaces}\label{proiettivo}
Let $V$ be a  real finite dimensional vector space endowed by a scalar product $\scalo$. Let $G\subset \mathrm{GL}(V)$ a closed reductive subgroup satisfying $G=K\exp (\liep)$, where $K=G\cap \mathrm{O}(V)$ is a maximal compact subgroup of $G$ and  $\liep=\lieg\cap \mathrm{Sym}(V)$. We denote by
$\parallel \cdot \parallel=\sqrt{\scalo}$. The $G$ action on $V$ induces, in a natural way, an action on the projective space $\mathbb P(V)$ given by
\[
G \times \mathbb P (V) \lra \mathbb P (V) \qquad (g,[v])\mapsto [gv].
\]
\begin{lemma}
The function
\[
\tilde \Psi : G \times \mathbb P (V) \lra \R \qquad (g,[x]) \mapsto \log \left(\frac{\parallel g x \parallel}{\parallel x \parallel} \right),
\]
is a Kempf-Ness function and the corresponding gradient map $\tilde \fun_\liep :\mathbb P (V) \lra \liep^*$ is given by
\[
\tilde \fun_\liep ([x])(\xi)=\frac{\langle \xi x , x \rangle}{\parallel x \parallel^2}.
\]
\end{lemma}
\begin{proof}
$P(1)$ and $P(2)$ are easy to check. Let $\xi \in \liep$ and let $f(t)=\tilde \Psi(x,\exp(t\xi))$. Then
\begin{gather*}
f'(t)=\frac{\langle \exp(t \xi ) \xi x, \exp(t\xi) x \rangle}{\langle \exp (t \xi) x , \exp (t \xi) x \rangle}\\
f''(t)= 2\frac{\langle \exp(t \xi ) \xi x, \exp(t\xi) \xi x \rangle^2 \langle \exp(t\xi)x, \exp (t\xi) x \rangle^2-\langle \exp (t\xi)\xi x, \exp(t\xi) x \rangle^2}{\langle \exp(t\xi) x , \exp (t \xi) x \rangle^2}.
\end{gather*}
By the Cauchy-Schwartz's inequality we have $f''(t) \geq 0$. Moreover, $f''(0)=0$ if and only if $\xi x$ and $x$ are linearly dependent and so if and only if $\exp(\R \xi ) \subset G_{[x]}$. Now,
\[
\begin{split}
\Psi(x,hg)&=\log \left(\frac{\parallel hg x \parallel}{\parallel x \parallel}\right) \\
          &=\log \left( \frac{\parallel hg x \parallel}{\parallel gx \parallel} \right) + \log \left( \frac{\parallel g x \parallel}{\parallel x \parallel} \right) \\
          &=\tilde \Psi(gx,h)+\tilde \Psi(x,g),
\end{split}
\]
proving the cocycle condition. Finally
\[
\begin{split}
\tilde \fun_\liep (x)(\xi) &=\desudtzero  \log \left(\frac{ \parallel \exp (t\xi) x ,\exp(t\xi) x \parallel}{\parallel x \parallel} \right)\\
                           &=\frac{\langle \xi x , x \rangle}{\parallel x \parallel^2}.
\end{split}
\]
concluding the proof.

\end{proof}
Let $\lia\subset \liep$ be an Abelian subalgebra. It is easy to check that $\tilde \fun_\lia=\pi_\lia \circ \tilde \fun_\liep$ is the gradient map associated to $A$. We denote by $\mathbb P(V)^A=\{y\in \mathbb P(V):\, A\cdot y=y\}$, the fixed points set.
\begin{teo}\label{convexity-projective-abelian}
Let $x\in \mathbb P(V)$. The map $\tilde \fun_\lia :\overline{A\cdot x} \lra \lia^*$ satisfies the following properties;
\begin{enumerate}
\item $\tilde \fun_{\lia}(A\cdot x)$ is diffeomorphic to an open convex subset of $\tilde \fun_{\lia}(x)+\lia_x^\perp$, its closure coincides with
  $\tilde \fun_\lia (\overline{A\cdot x})$, it is a polytope $P$ and it is the convex
hull of $\tilde \fun_\lia (\mathbb P(V)^A\cap \overline{A\cdot x})$.
\item $\fun_\lia :\overline{A\cdot x} \lra P$ is an homomorphism. Moreover,  any face of $P$ is homeomorphic to a closure of a unique $A$ orbit contained in $\overline{A\cdot x}$.
\end{enumerate}
\end{teo}
\begin{proof}
The first statement is well-known. A proof using the Kempf-Ness function is given in  \cite{bilio-ghigipr}. Item $(b)$ is known in the complex setting \cite{atiyah-commuting}.

Let $\sigma \subset P$ be a face of $P$. It is
  an exposed face, so there exists $\xi\in \lia$ such that
  \begin{gather*}
    \sigma=\{z\in P: \langle z,\xi \rangle = \max_{y\in \fun_\lia (\overline{A\cdot x})}
    \langle y,\xi \rangle \}.
  \end{gather*}
    Since for any $y\in \mathbb P(V)$,  the function $t\mapsto \fun_\lia (\exp(t\xi) y) (\xi)$ is monotone, it follows that $\fun_\lia^{-1} (\sigma) \subset \mathbb P(V)^{\Gamma}$, where $\Gamma=\exp(\R \xi)$ and so $\fun_\lia^{-1} (\sigma)$ is $A$-invariant.  In the sequel we denote by $\fun_\lia^\xi$ the function $x\mapsto \fun_\lia (x)(\xi)$ which is non-degenerate in the sense of  Morse-Bott \cite{atiyah-commuting,guillemin-sternberg-convexity-1,heinzner-schwarz-Cartan,heinzner-schwarz-stoetzel}, see also \cite{biliotti-ghigi-heinzner-2} p. $589$ for a proof.  Applying the Linearization Theorem \cite{heinzner-schwarz-Cartan,heinzner-schwarz-stoetzel}, for any $x\in \mathbb P(V)$ the limit $\lim_{t\mapsto +\infty} \exp (t \xi ) x$ exists and the set of all points of $\mathbb P(V)$ for which $\lim_{t\mapsto +\infty} \exp (t \xi )\cdot $ lies on a given critical manifold $C$ forms a submanifold and they are called unstable manifolds. We denote by $\phi_\infty$ the limit map. Let  $\xi\in \liep$ and let $\lambda_1>\dots> \lambda_k$ be its eigenvalues.  We denote by $V_1,\dots, V_k$ the corresponding eigenspaces.  In view of the orthogonal decompositions $V=V_1\oplus\dots\oplus V_{k}$, the
critical points of $\fun_\lia^\xi$ are given by $\mathds{P}(V_1)\cup\dots\cup \mathds{P}(V_{k})$ and the corresponding unstable manifolds are given by:
$$
W_1^\xi=\mathds{P}^{n}(\mathds{R})\setminus \mathds{P}(V_2\oplus\dots\oplus V_k),
$$
$$
W_2^\xi=\mathds{P}(V_2\oplus\dots\oplus V_k)\setminus \mathds{P}(V_3\oplus\dots\oplus V_k),
$$
$$
\vdots
$$
$$
W_{k-1}^\xi=\mathds{P}(V_{k-1}\oplus V_k)\setminus \mathds{P}(V_k).
$$
$$
W_k^\xi=\mathds{P}(V_k).
$$
(see \cite{bilio-zedda}).
Assume that $x\in W_j^\xi$ for some $1\leq j\leq k$. Then  $x_\infty=\lim_{t\mapsto +\infty} \exp(t\xi) x \in \mathbb P(V_j)$ and $\overline{A\cdot x}\subset \mathbb P(  (V_j\oplus \cdots \oplus V_k)$. Since $\fun_\lia^{\xi}$ restricted to
$\mathbb P(  V_j\oplus \cdots \oplus V_k)$ has $\mathbb P(V_j)$ as a unique maximum it follows $x_\infty \in \fun_\lia^{-1} (\sigma)$. Therefore
$
\overline{A \cdot x_\infty} \subset \fun_\lia^{-1}(\sigma).
$
Let $a_n$ be a sequence of elements of $A$ such
that $a_n \cdot x \mapsto \theta \in \fun_\lia^{-1}(\sigma)$. Since
\[
\phi_\infty : W_k^\xi \lra \mathbb P(V_j) \qquad y \mapsto \lim_{t\mapsto +\infty} \exp (t \xi ) y
\]
is smooth, it follows
\begin{equation*}
  \theta=\lim_{n\mapsto \infty } \phi_\infty (a_n \cdot x)
  =\lim_{n\mapsto \infty} a_n\cdot x_\infty,
\end{equation*}
and so it lies in $\overline{A\cdot x_\infty}$. Therefore
\[
\fun_\lia^{-1}(\sigma)=\overline{A\cdot x_\infty}.
\]
Now, again applying item $(1)$, the map $\fun_\lia : A \cdot x_\infty \lra \mathrm{relint} \sigma$ is a diffeomorphism, proving item $(b)$.
\end{proof}
The following result is a direct consequence of the above Theorem.
\begin{cor}
Let $x\in \mathbb P(V)$. Then
\begin{enumerate}
\item $\overline{A\cdot x} \setminus A\cdot x$ is a finite union of $A$ orbits.
\item $u\in \overline{A\cdot x}$ if and only if there exists $\xi \in \lia$ and $a\in A$ such that $\lim_{t\mapsto +\infty } \exp (t\xi) a x=u$.
\end{enumerate}
\end{cor}
Note that the second item provides a Hilbert-Mumford criterion for Abelian groups acting on $\mathbb P(V)$. Now we completely describe the image of the gradient map along $A$ orbits. In the sequel we denote by $\pi: V\setminus\{0\} \lra \mathbb P(V)$ the natural projection. Since $A$ is Abelian, there exists an orthonormal basis $\{v_1,\ldots, v_n\}$ of $V$ such that diagonalize simultaneously any element of $A$. Therefore
\[
\tilde \fun_\lia ([x_1 v_1 + \cdots + v_n e_n])(\xi)=\frac{x_1^2 \alpha_1 (\xi) +\cdots + x_n^2 \alpha (\xi) }{x_1^2 + \cdots + x_n^2}
\]
If $I\subset \{1,\ldots,n\}$, we denote by $P_I=\mathrm{Conv}(\alpha_k:\, k \in I)$ the  polytope generated by $\alpha_k$, as $k$ runs in $I$.
\begin{teo}\label{image}
Let $x\in \mathbb P(V)$ and let $\tilde x \in V$ such that $\pi(\tilde x)=x$. Let $I=\mathrm{supp}_{\tilde x}$. Then
\begin{enumerate}
\item $\tilde \fun_\lia (\overline{A\cdot x})=P_I$, where $P_I=\mathrm{Conv}(\alpha_i:\, i\in I)$.
\item $A\cdot x$ is closed if and only if $x\in \mathbb P(V)^A$;
\end{enumerate}
\end{teo}
\begin{proof}
By definition of $\tilde \fun_\lia$ it follows $\tilde \fun_{\lia} (\overline{A\cdot x})\subset P_I$.  It is well known that $P_I$ is the convex hull of its extremal points \cite{schneider-convex-bodies}. Since $P_I$ is generated by $\alpha_k$, as $k$ runs on $I$, the extremal points of $P_I$ are contained in $\{\alpha_j : j\in I\}$. Assume that $\alpha_s$ is an extremal point of $P_I$ for some $s\in I$. Let $J=\{r\in I$ such that $\alpha_r=\alpha_j\}$. Since any face of $P_I$ is exposed, $\{\alpha_s \}$ is an exposed face and so  there exists $\xi \in \lia$ and $c>0$ such that $\langle \xi, \alpha_i \rangle <c$ for $i\notin J$ and $\langle \xi, \alpha_s \rangle =c$ for $s\in J$. Therefore
\[
\begin{split}
\tilde \fun_\lia (\exp(t\xi) x)&= \sum_{i \in I} \frac{x_i^2 e^{2t \langle \alpha_i , \xi \rangle}\alpha_i }{ x_i^2 e^{2t \langle \alpha_i , \xi \rangle}} \\
                                &=\frac{ \sum_{m \in J} x_m^2 \alpha_s + \sum_{i\in I\setminus\{J\}} x_i^2 e^{2t(\langle \alpha_i ,\xi \rangle -c)}}{\sum_{m\in J} x_m^2 + \sum_{i\in I\setminus\{J\}} x_i^2  e^{2t(\langle \alpha_i ,\xi \rangle -c)}}.
\end{split}
\]
Taking the limit $t\mapsto +\infty$, it follows $\alpha_s \in \tilde \fun_\lia (\overline{A\cdot x})$. This holds for any extremal point of $P_I$ and so $\tilde \fun_\lia (\overline{A\cdot x})=P_I$ proving the first item.

Since  $\tilde \fun_\lia (\overline{A\cdot x})=\overline{\tilde \fun_\lia (A\cdot x)}$, $A\cdot x$ is closed if and only if $\tilde \fun_\lia (A\cdot x)$ is closed. Since the image of any $A$-orbit is the relative interior of a polytope, it follows that $A\cdot x$ is closed if and only if $A\cdot x=x$.
\end{proof}
We now compute the image of the closure of a $G$ orbit under the gradient map.
\begin{teo}\label{convex-orbit-general-projective}
Let $x\in \mathbb P (V)$ and let $\tilde x \in V$ such that $\pi(\tilde x)=x$. There exists $v\in G\cdot \tilde x$ such that $\tilde \fun_{\lia}(\overline{G\cdot x})=\fun_{\lia}(\overline{A\cdot \pi(v)} )=P_I$, where $I=\mathrm{supp}_{v}$, and so  it is a polytope. Moreover, the set $\mathrm{Conv}(\tilde \fun_\liep (\overline{G\cdot x}))=KP_I$ and so any face of  $\mathrm{Conv}(\tilde \fun_\liep (\mathrm{G\cdot x}))$  is exposed.
\end{teo}
\begin{proof}
Let $\tilde x \in V$ such that $\pi(\tilde x)=x$.
By Theorem \ref{convex-orbit-general} it follows that $\fun_\lia (\overline{G\cdot \tilde x})=\fun_{\lia}(\overline{A\cdot \tilde v})=C_I$, where $I=\mathrm{supp}_{\tilde v}$. By the above corollary $\tilde \fun_\lia (\overline{A \cdot v})=P_I$, where $v=\pi (\tilde v)$. Since for any $\tilde y\in \overline{G\cdot \tilde x}$, we get  $\fun_\lia (\overline{A\cdot \tilde y})\subset C_I$, it follows $\tilde \fun_\lia (\overline{A\cdot y})\subset P_I$, for any $y\in G\cdot x$ and so the first part is proved.

Set $E= \mathrm{Conv}(\tilde \fun_\liep (\overline{G\cdot x}))$. By Lemma $7$ in \cite{gichev-polar}  we get $\pi_\lia (E)=\tilde \fun_\lia (\overline{A\cdot \pi(v)})$ and so $E=K \tilde \fun_\lia(\overline{A\cdot v})=C P_I$. By Theorem $0.3$ in \cite{bgh-israel-p}, any face of $E$ is exposed.
\end{proof}
\section{Hilbert-Mumford criterion for reductive groups}
In this section we prove the Hilbert-Mumford criterion for real reductive groups. We use in a different context original ideas from \cite{georgula}.

Let $G\subset \mathrm{GL}(V)$ a closed reductive subgroup satisfying $G=K\exp (\liep)$, where $K=G\cap \mathrm{O}(V)$ is a maximal compact subgroup of $G$ and  $\liep=\lieg\cap \mathrm{Sym}(V)$.
The function
\[
\tilde \Psi : G \times \mathbb P (V) \lra \R \qquad (g,[x]) \mapsto \log \left(\frac{\parallel g x \parallel}{\parallel x \parallel} \right),
\]
is a Kempf-Ness function and the corresponding gradient map $\tilde \fun_\liep :\mathbb P (V) \lra \liep^*$ is given by
\[
\tilde \fun ([x])(\xi)=\frac{\langle \xi x , x \rangle}{\parallel x \parallel^2}.
\]
We may fix $\mathrm{Ad}(K)$-invariant scalar product $\scalo_{\lieg}$ on $\lieg$ such that $\lieg =\liek \oplus \liep$ is an orthogonal splitting. Hence we may think the gradient map as a $\liep$-valued map,
\[
\tilde \fun_\liep : \mathbb P(V) \lra \liep.
\]
We define
\[
f:\mathbb P(V) \lra \R, \qquad [x] \mapsto \frac{1}{2}\langle \tilde \fun_\liep ([x]),\tilde \fun_\liep ([x])\rangle_\lieg.
\]
In the sequel if $\xi \in \lieg$, we denote by $\xi^\# ([x]):=\desudtzero \exp(t\xi) [x]$ the vector field induced by the $G$ action.
\begin{lemma}
The function $f$ is analytic and its gradient is given by $\nabla f (x)=\campo (x)$.
\end{lemma}
\begin{proof}
We identify $V$ with $\R^n$. Assume that $G=\mathrm{SL}(n)$. Since
\[
\tilde \fun_{Sym(n)} ([x])=\frac{xx^T}{\parallel x \parallel^2} - \frac{1}{2n}Id,
\]
$\tilde \fun_{Sym(n)}$ is a polynomial and so it is analytic. If $G\subset \mathrm{SL}(V)$, then $\tilde \fun_\liep =\pi_\liep \circ \tilde \fun_{Sym(n)}$, where $\pi_\liep$ is the orthogonal projection,  and so it analytic as well. The second part of the statement is easy to check.
\end{proof}
The negative gradient flow of $f$ throughout $y\in \mathbb P(V)$ is then solution of the differential equation
\begin{equation}\label{equazione}
\left\{\begin{array}{l}
\dot {x}(t)=-\campo (x(t))\\
x(0)=y \end{array}\right.
\end{equation}
Since $\mathbb P(V)$ is compact, the solution is defined in all $\R$.
\begin{lemma}\label{equazione-sollevata}
Let $g:\R \lra G$ be the unique solution of the differential equation
\[
\left\{
\begin{array}{l}
g^{-1} (t) \dot{g}(t)=\campo (x(t))\\
g(0)=e
\end{array}
\right. .
\]
Then $x(t)=g^{-1}(t) y$.
\end{lemma}
\begin{proof}
The solution $g$ is defined in all $\R$ (see \cite{kobayashi} p. $69$). Since $\dot{g^{-1}}=-g^{-1}\dot{g}g^{-1}$, it follows
\[
\dot{x}(t) =-g^{-1} \dot{g} g^{-1} y=- \campo (x(t)),
\]
and so the result is proved.
\end{proof}
The following Theorem arises from Lojasiewicz gradient inequality, see \cite{birestone,losa}. A proof is given in \cite[Theorem 3.3 p.$14$]{georgula}.
\begin{teo}\label{loja}
In the above assumption, the limit $\lim_{\mapsto +\infty} x(t)=x_\infty$ exists. Moreover, there exists positive constants $\alpha,c,\beta$, $\frac{1}{2}<\gamma <1$ and $T>0$ such that for any $t > T$ we have
\[
\begin{split}
d(x(t),x_\infty)&\leq \int_{t}^{\infty} \parallel \dot x (s) \parallel \mathrm{d} s \\
            &\leq \frac{\alpha}{1-\gamma}(f(x(t))-f(x_\infty))^{1-\gamma} \\
            &\leq \frac{c}{(t-T)^\beta}.
\end{split}
\]
\end{teo}
The following Theorem is a consequence of the Stratification Theorem in \cite{heinzner-schwarz-stoetzel} and results proved in \cite{jabo,lerman}.
\begin{teo}\label{limite}
Let $y\in \mathbb P(V)$ and let $x:\R \lra \mathbb P(V)$ be the solution of \ref{equazione}. Let $x_\infty=\lim_{t\mapsto +\infty} x(t)$. Then
\[
\parallel \tilde \fun_\liep (x_\infty) \parallel=\mathrm{inf}_{z\in \overline{G \cdot y}} \parallel \tilde \fun_\liep (z) \parallel.
\]
Moreover, if $z\in \overline{G \cdot y}$ satisfies $\parallel \tilde \fun_\liep (z)\parallel=\mathrm{inf}_{z\in \overline{G \cdot y}} \parallel \tilde \fun_\liep (z) \parallel$, then $z\in K\cdot x_\infty$.
\end{teo}
\begin{proof}
Let $\mathcal C_\liep$ be the set of critical points of $f$. Set $\mathcal B_\liep=\tilde \fun_\liep (\mathcal C_\liep)$. Since $f$ is $K$-invariant, the sets $\mathcal C_\liep$ and $\mathcal B_\liep$ are $K$-invariant.
By the Stratification Theorem \cite{heinzner-schwarz-stoetzel}, there exists $\beta \in \mathcal B_\liep$ such that $\tilde \fun_\liep (\overline{G\cdot y})  \cap \mathcal B_\liep=K\beta$ and
$
\mathrm{inf}_{z\in \overline{G \cdot y}} \parallel \tilde \fun_\liep (z) \parallel=\parallel \beta \parallel.
$
Since $x_\infty$ is a critical point of $f$ it follows $
\mathrm{inf}_{z\in \overline{G \cdot y}} \parallel \tilde \fun_\liep (z) \parallel=\parallel \tilde \fun_\liep (x_\infty) \parallel.$
By  Theorem $5.1$ in  \cite{jabo}, $G\cdot y$ collapses to a single $K$ orbit under the negative gradient flow of $f$.  Let $C$ be the connected component of the critical set of  $f$ corresponding to $\beta$ and let $S_{\beta}$ the corresponding stratum. In \cite{lerman}, in the complex setting, the author proves that the map
\[
\phi_\infty: S_\beta \lra C \qquad \phi_\infty (x)=x_\infty,
\]
is a continuous retraction. This result follows by the Lojasiewicz gradient inequality (see Lemma 2.3. p.$124$). Therefore the same holds in our situation.

Let $g_n y \mapsto r$ be such that
\[
\mathrm{inf}_{z\in \overline{G \cdot y}} \parallel \tilde \fun_\liep (z) \parallel=\parallel \tilde \fun_\liep (r) \parallel.
\]
Then
\[
\lim_{n \mapsto +\infty} \phi_\infty (g_n y)=\phi_\infty (r)=r,
\]
and so $r$ belongs to $K \cdot x_\infty$.
\end{proof}
We define on $G$ the left-Riemannian metric which agree with  $\scalo_{\lieg}$ on the tangent space $T_e G$. This metric is $K$-invariant with respect the right $K$-action. Let $\Phi_x: G \lra \R$, defined as
\[
\Phi_x (g)=\log \left(\frac{\parallel g^{-1} x \parallel}{\parallel x \parallel} \right),
\]
\begin{lemma}\label{scambio-gradienti}
The differential of $\Phi_x$ is given by $(\mathrm{d} \Phi_x )_g (v)=-\langle \tilde \fun_\liep (g^{-1} x), \mathrm{d} L_{g^{-1}}(v) \rangle$. Therefore $\nabla \Phi_x (g)=v_x(g)$
where $v_x(g)=-\mathrm{d} L_g (\tilde \fun_\liep (g^{-1} x))$,
\end{lemma}
\begin{proof}
Let $g\in G$ and let $X\in \lieg$. Then
\[
(\mathrm{d} \Phi_x )_g ( \mathrm{d} L_g (X))=\frac{-\langle X (g^{-1} x), g^{-1} x \rangle}{\parallel g^{-1} x \parallel^2}.
\]
If $X\in \liep$ then $(\mathrm{d} \Psi_x )_g ( \mathrm{d} L_g (X))=-\langle \tilde \fun_\liep (g^{-1} x), X \rangle_\lieg$. If $X\in \liek$, then
\[
0=(\mathrm{d} \Phi_x )_g ( \mathrm{d} L_g (X))=\langle \tilde \fun_\liep (g^{-1}x), X \rangle_\lieg.
\]
This means
\[
\begin{split}
(\mathrm{d} \Phi_x )_g ( \mathrm{d} L_g (X))&=-\langle \tilde \fun_\liep (g^{-1}x), X \rangle_\lieg \\
                                           &=-\langle \mathrm{d} L_g ( \fun_\liep (g^{-1}x) ), \mathrm L_g (X) \rangle,
\end{split}
\]
concluding the proof.
\end{proof}
Define $\Theta_x : G \lra G(x)$ as follows
\[
\Theta_x (g)=g^{-1} x.
\]
\begin{lemma}
The map $\Theta_x$ intertwines the gradient of $\Phi_x$ and $\nabla f$.
\end{lemma}
\begin{proof}
Let $\xi \in \liep$. Since
\[
\Theta_x (g\exp(t\xi))=\exp(-t \xi ) g^{-1} x,
\]
we get
\[
(\mathrm d \Theta_x )_g (\mathrm{d} L_g (\xi))=-\xi^{\#}(g^{-1}x),
\]
and so the result is proved.
\end{proof}
Since $\Phi_x$ is $K$-invariant, it descends to a smooth map $\Phi_x : G/K \lra \R$, that we also denote by $\Phi_x$. A proof of the next Lemma is given in \cite{georgula} (Lemma A.3 p.$140$).
\begin{lemma}\label{dietmar}
Let $F:G/K \lra \R$ be a smooth function that is convex along geodesics. Let
$c_0,c_1: \R \lra M$ be the negative gradient flow of $F$ and let $\rho(t)=d_M (c_0 (t),c_1 (t))$. Then $\rho(t)$ is nonincreasing function.
\end{lemma}
\begin{teo}\label{Kemf-Ness-Morse}
Let $\Phi_x :G/K \lra \R$. Then
\begin{enumerate}
\item $\Phi_x$ is a Morse-Bott function and it is convex along geodesics;
\item The critical set, possibly empty, is the submanifold
\[
\{gK \in G/K:\, \tilde \fun_\liep (g^{-1} x)=0\}.
\]
\item If $c:I \lra M$ is a negative gradient flow of  $\Phi_x$, then
\[
\lim_{t\mapsto +\infty} \Phi_x (c(t))=\mathrm{inf}_{x\in G/K} \Phi_x
\]
\end{enumerate}
\end{teo}
\begin{proof}
By Lemma $6$ in \cite{bilio-zedda}, see also Lemma $2.19$ p. $1115$, $\Phi_x$ is convex along geodesics. By Proposition $9$ p.2191 it follows that a critical points of $\Phi_x$ are the elements $\pi(g) \in G/K$ such that $\tilde \fun_\liep (g^{-1} x)=0$. If $gK=\pi(g)$ is a critical point, then
\[
\mathrm{Hess}(\Phi_x)(\mathrm{d} L_g (\xi), \mathrm{d} L_g (\xi))= \frac{\mathrm{d^2}}{\mathrm{dt}^2 }\bigg \vert_{t=0}
      \Phi_x (g\exp(t\x)) =  \frac{\mathrm{d^2}}{\mathrm{dt}^2 }\bigg \vert_{t=0} \log \left( \parallel \exp(- t \xi ) g^{-1} x \parallel \right) \geq 0,
\]
and it is $0$ if and only if $\exp(t\xi) \subset G_{g^{-1}x}$. Let $g\in \mathrm{Crit} \Phi_x$. By Hadamard-Cartan Theorem,
the map
\[
\liep \lra G/K, \qquad \xi \mapsto g\exp(\xi),
\]
is a diffemorphism. Theretofore $g\exp(\xi) \in \mathrm{Crit} \Phi_x$ if and only if $\tilde \fun_\liep (\exp (-\xi) g^{-1} x)=0$. Since
$\tilde \fun_\liep (g^{-1} x)=0$, by Proposition \ref{compatible} it follows that $\exp(t\xi) \in G_{g^{-1}x}$. This implies
\[
\mathrm{Crit} \Phi_x =\{\pi(g\exp(t\xi)):\, \exp(t\xi) \in G_{g^{-1} x} \}
\]
and so it is submanifold and the Kernel of the Hessian. Therefore $\Phi_x$ is a Morse-Bott function.

The gradient of $\Psi_x :G \lra \R$ is given by $\nabla \Psi_x (\pi(g))=\mathrm{d} \pi_g (v_x (g))$. Hence  the negative gradient flow of $\Psi_x : G/K \lra \R$ lifts to the negative gradient flow of $\Psi_x:G \lra \R$. By Lemma \ref{scambio-gradienti} the map $\Theta_x$ intertwines the gradient of $\Psi_x$ with $\nabla f$. Therefore the negative gradient flow of $\Psi_x :G \lra \R$  satisfies the differential equation \ref{equazione}. Vice-versa if $g:\R \lra G$ satisfies the equation \ref{equazione}, then one may prove that  $\pi \circ g$ is the negative gradient flow of $\Psi_x : G/K \lra \R$.

Let $c_1,c_2:\R \lra M$ be negative gradient flow  of $\Phi_x$. Then there exist $g_0,g_1:\R \lra G$ solution of \ref{equazione} such that $c_0=\pi \circ g_0$ and $c_1 =\pi \circ g_1$.  Since $G=K\exp(\liep)$, there exist $\xi:\R \lra \liep$ and $k:\R \lra K$ such that
$
g_1 (t)=g_0 (t)\exp(\xi(t))k(t).
$
Write
\[
H:\R \times \R \lra G/K, \qquad H(t,s)=\pi(g_0(t)\exp(s\xi(t)).
\]
In the sequel we denote by $H_t (s)=H(t,s)$. The curve $s\mapsto H_t(s)$ is the unique geodetic joining $c_0(t)$ and $c_1(t)$. By Lemma \ref{dietmar} the function
\[
\rho(t)=d_{G/K}(c_0(t),c_1(t))=\parallel \xi(t) \parallel,
\]
is nonincreasing. Assume $\mathrm{Crit} \Phi_x$ is not empty. Hence we may assume $g_0 (0) \in \mathrm{Crit} \Phi_x$ and so the curve $c_0$ is constant. Since $\rho$ is nonincreasing, the image $c_1$ is contained in a compact subset. This implies, keeping in mind that $\Phi_x$ is Morse-Bott,   the limit $\lim_{t\mapsto +\infty} c_1 (t) \in \mathrm{Crit} \Phi_x$ and so item $(c)$ holds. In particular, every negative gradient flow converges to a critical point and so $\Phi_x$ has a global minimum.  Now, assume that $\Phi_x$ does not have any critical point. Assume by contradiction
\[
\lim_{t\mapsto +\infty} \Phi (c_0 (t)) =a > \mathrm{Inf}_{G/K} \Phi_x.
\]
Hence $\Phi (c_0 (t))$ is bounded from below. We may choose $c_0$ such way $\Phi_x (c_1 (0))<a$. By Lemma \ref{dietmar},  $\rho$ is nonincreasing and so there exists $C>0$ such that $\rho (t)=\parallel \xi (t) \parallel \leq C$. Hence
\[
\begin{split}
\desus \Phi_x (H_t (s))&=(\mathrm{d} \Phi_x)_{c_0 (t)}(\dot{H}_t (0)) \\
                                &=-\langle \tilde \fun_\liep (g_0(t)^{-1}x),\xi(t)\rangle \\
                                &-\geq \parallel  \tilde \fun_\liep (g_0(t)^{-1}x) \parallel \parallel \xi (t) \parallel \\
                                &-\geq C \parallel  \tilde \fun_\liep (g_0(t)^{-1}x) \parallel.
\end{split}
\]
Since for $t$ fixed, $H_t (s)$ is a geodesic, it follows that the function $s \mapsto \Phi_x (H_t (s))$ is convex and so
its derivative $\deriva{s}{s} \Phi_x (H_t (s))$ increases. Hence
\[
\begin{split}
\Phi_x (c_1 (t))&=\Phi_x (H_t (1))) \\
                    &=\Phi_x(H_t (0)) + \int_0^1 \deriva{s}{s} \Phi_x (H_t (s)) \mathrm{ds} \\
                    & \geq \Phi_x (c_0 (t))-C  \parallel \tilde \fun_\liep (g_0(t)^{-1}x) \parallel.
\end{split}
\]
Now, the function $\Phi \circ c_1$ is bounded and $\deriva{t}{t} (\Phi_x \circ c_0)=- \parallel\fun_\liep (g_0(t)^{-1}x) \parallel$. Therefore,  there exists a sequence $t_i \mapsto +\infty$ such that $\deriva{t}{t=t_i} (\Phi_x \circ c_0)=- \parallel\fun_\liep (g_0(t_i)^{-1}x) \parallel$ goes to $0$.
This implies
\[
\lim_{t_i \mapsto +\infty} \Phi_x (c_1 (t_i))\geq  \Phi_x (c_0 (t_i))\geq a,
\]
a contradiction since $(\Phi_x (c_o(t))<a$ and so $\lim_{t_i \mapsto +\infty} \Phi_x (c_1 (t_i))<a$.
\end{proof}
Now, we are able to prove the Hilbert-Mumford criterion for real reductive groups. We start recalling a well-known numerical criterion, a proof is given in \cite{bilio-zedda}, and some technical lemmata.
\begin{teo}\label{numerical-criteria}
Let $x\in \mathbb P(V)$. Then $x$ is semistable if and only if for any $\xi \in \liep$, $\lambda (x, \xi)\geq 0$;
\end{teo}
\begin{lemma}\label{lemma1}
Let $x\in \mathbb P(V)$ and let $\tilde x\in V$ such that $\pi(\tilde x)=x$. Let $\alpha_1 > \cdots >\alpha_k $ be the eigenvalues of $\xi$. We denote by $V_i$ the corresponding eigenspaces. Write $\tilde x=v_1+\cdots+v_k$. Then
$
\lambda(x,\xi)=\alpha_j,
$
where $j=\mathrm{min}   \{1,\ldots k\}$ such that $v_j\neq 0$.
Moreover, $\lim_{t\mapsto +\infty} \exp(t\xi ) \tilde x=0$ if and only if $\alpha_j<0$.
\end{lemma}
\begin{proof}
\[
\begin{split}
\lambda(x,\xi)&=\lim_{t\mapsto +\infty} \desudt \tilde \Psi(x,\exp(t\xi)) \\
              &=\lim_{t\mapsto +\infty} \frac{\alpha_j e^{2\alpha_i t} \parallel v_j \parallel^2 + \cdots + \alpha_k e^{2\alpha_k t} \parallel v_k \parallel^2 }{ e^{2\alpha_j t } \parallel v_j \parallel^2 + \cdots + e^{2\alpha_k t} \parallel v_k \parallel^2} \\
              &=\lim_{t\mapsto +\infty} \frac{\alpha_j  \parallel v_j \parallel^2 + \alpha_{j+1} e^{2(\alpha_{j+1} -\alpha_j) t} \parallel v_{j+1} \parallel^2+ \cdots + \alpha_k e^{2(\alpha_k -\alpha_j) t} \parallel v_k \parallel^2 }{ \parallel v_j \parallel^2 + e^{2(\alpha_{j+1}-\alpha_j)t} \parallel v_{j+1} \parallel^2 \cdots + e^{2(\alpha_k-\alpha_j) t} \parallel v_k \parallel^2} \\
              &=\alpha_j
\end{split}
\]
Since
\[
\parallel \exp(t\xi) \tilde x \parallel^2 = \sum_{i=1}^k e^{2t\alpha_i} \parallel v_i \parallel^2,
\]
it follows that $\exp(t\xi) \tilde x \mapsto 0$ if and only if $\alpha_j <0$. Otherwise the limit does not exist or $\exp(t\xi) \in G_{\tilde x}$.
\end{proof}
\begin{lemma}\label{semistable-projecive-vector}
Let $x\in \mathbb P(V)$ be a semistable point. Let $\tilde x \in V$ be such that $\pi(\tilde x)=x$. Then $0\notin \overline{G\cdot \tilde x}$.
\end{lemma}
\begin{proof}
Let $s:\R \lra \mathbb P(V)$ and let $g:\R \lra G$ so that $s(t)=g(t)^{-1} x$ is the solution of \ref{equazione}. By Theorem \ref{loja}, the limit $\lim_{t\mapsto +\infty} s(t)=x_\infty$ exists and by Theorem \ref{limite} it satisfies
\[
\parallel \tilde \fun_\liep (x_\infty) \parallel=\mathrm{inf}_{y \in \overline{G(x)}} \parallel \tilde \fun_\liep (y) \parallel.
\]
Since $x$ is semistable it follows that $\parallel \tilde \fun_\liep (x_\infty) \parallel=0$. By the Lojasiewicz gradient inequality for $f$, there exist positive constant $\alpha,\beta$ and $\frac{1}{2} < \gamma < 1$ such that
\[
\parallel \tilde \fun_\liep (x) \parallel^2 \leq 2f(x)\leq 2 f(x)^\gamma \leq C \parallel \grad f \parallel = C \parallel \campo (x) \parallel.
\]
Therefore
\[
\parallel \tilde \fun_\liep (s(t)) \parallel^2 =- \frac{\mathrm{d}}{\mathrm{dt}}\Phi_{[x]} \circ c (t)  \leq C \parallel \dot{s}(t) \parallel,
\]
where $c(t)=\pi \circ g(t)$ is the negative gradient flow of $\Phi_x:G/K \lra \R$.
By Theorem \ref{loja} the function $\parallel \dot{s} (t) \parallel$ is integrable over $\R^{+}$, and so the limit
\[
\lim_{t\mapsto +\infty} \Phi_{[x]} \circ c (t)=a \in \R.
\]
By Theorem \ref{Kemf-Ness-Morse} the function $\Phi_x$ is bounded below. Since $\Phi_{[x]}(g)=\left(\frac{\parallel g^{-1} \tilde x \parallel}{\parallel \tilde x \parallel} \right)$, we get  $0\notin \overline{G\cdot \tilde x}$ concluding the proof.
\end{proof}
\begin{teo}[Hilbert-Mumford criterion for reductive groups]\label{hmr}
Let $y\in V$ and assume that $0\in \overline{G\cdot y}$. Then there exists $\xi \in \liep $ such that $\lim_{t\mapsto +\infty} \exp(t\xi ) y= 0$.
\end{teo}
\begin{proof}
Let $y\in V$ such that $0\in \overline{G\cdot y}$. By Lemma \ref{semistable-projecive-vector} the point $\pi(y)$ is not semistable in $\mathbb P(V)$. By Theorem \ref{numerical-criteria} there exists
$\xi \in \liep $ such that $\lambda_x (\pi(y),\xi)<0$. By Lemma \ref{lemma1} we get $\lim_{t\mapsto +\infty} \exp(t\xi ) y=0$ concluding the proof.
\end{proof}


\begin{thebibliography}{10}
\bibitem{atiyah-commuting} M.~F. Atiyah.  \newblock Convexity and
  commuting {H}amiltonians.  \newblock {\em Bull. London Math. Soc.},
  14 (1):1--15, 1982.
%
\bibitem{azad-loeb-bulletin} H.~Azad and J.-J. Loeb.  \newblock
  Plurisubharmonic functions and the {K}empf-{N}ess theorem.
  \newblock {\em Bull. London Math. Soc.}, 25(2):162--168, 1993.
%
\bibitem{bhc} A. Borel,  Harish-Chandra,
\newblock Arithmetic subgroups of algebraic groups,
\newblock \emph{Ann. of Math.} (2) 75, 485-535 (1962).
%
\bibitem{borel-book} A. Borel,
\newblock \emph{Linear Algebraic groups}
\newblock Second edition. Graduate Texts in Mathematics, 126.
\newblock Springer-Verlag, New York, p. 288 (1991).
%
\bibitem{bt} A. Borel, J. Tits,
\newblock Groupes r\'eductifs,
\newblock \emph{Inst. Hautes \'Etudes Sci. Publ. Math.} 27, 55-150 (1965).
%
\bibitem{borel} A. Borel,
\newblock Lie groups and linear algebraic groups I. Complex and Real groups,
\newblock AMS, I.P. Stud. Adv. Math. vol 37, Amer. Math. Soc., Providence, RI, 1-49, (2006).
%
\bibitem{vergne}
N. Berline, M. Vergne,
\newblock \emph{Hamiltonian  manifolds  and  the  moment  map},
\newblock preprint http://nicole.berline.perso.math.cnrs.fr/cours-Fudan.pdf.
%
\bibitem{biliotti-ghigi-AIF} L.~Biliotti and A.~Ghigi.  \newblock
  Homogeneous bundles and the first eigenvalue of symmetric spaces.
  \newblock {\em Ann. Inst. Fourier (Grenoble)}, 58(7):2315--2331,
  2008.
%
\bibitem{biliotti-ghigi-American} L.~Biliotti and A.~Ghigi.  \newblock
  Satake-{F}urstenberg compactifications, the moment map and
  {$\lambda_1$}.  \newblock {\em Amer. J. Math.}, 135(1):237--274,
  2013.
%
\bibitem{biliotti-ghigi-heinzner-1-preprint}
L.~Biliotti, A.~Ghigi, and P.~Heinzner.
\newblock {C}oadjoint orbitopes.
\newblock \emph{Osaka J. Math.} 51 (4) (2014), 935-968.
%
\bibitem{biliotti-ghigi-heinzner-2} L.~Biliotti, A.~Ghigi, and
  P.~Heinzner.  \newblock {P}olar orbitopes.  \newblock {\em
    Comm. Anal. Geom.} 21 (3): 579--606, (2013).
%
\bibitem{bgh-israel-p} L.~Biliotti, A.~Ghigi, and P.~Heinzner.
  \newblock Invariant convex sets in polar representations.
 \newblock {\em Israel J. Math.} 213, (2016), 423-441
%
\bibitem{bgs} L. Biliotti, A. Ghigi.
Stability of measures on K\"ahler manifolds.
\newblock {\em Adv. Math.} 317, (2017) 1108-1150.
%
\bibitem{bilio-zedda}
L. Biliotti, M. Zedda.
\newblock Stability with respect to actions of real reductive Lie groups
\emph{Ann. Mat. Pura Appl.} (4), 196, (6), 2185-2211, 2017.
%
\bibitem{biliotti-raffero} L. Biliotti, A. Raffero.
\newblock Convexity theorems for the gradient map for the gradient map on probability measures,
\emph{Complex Manifolds} 5 (2018) 133-145.
%
\bibitem{bilio-ghigipr}
L. Biliotti, A. Ghigi.
\newblock Remarks on the abelian convexity theorem,
\emph{Proc. Amer. Math. Soc.} 146  (12) (2018) 5409-5419.
%
\bibitem{birestone}
E. Birestone, P. Milman,
\newblock Semianalytic and subanalytic sets,
\newblock \emph{Inst. Hautes Etudes Sei Publ Math} 67, 4--42, (1988).
%
%
\bibitem{Bir}
D. Birkes,
\newblock  Orbits of linear algebraic groups,
\newblock \emph{Ann. of Math.} (2) 93, 459-475 (1971).
%
%
\bibitem{fb} B\"ohm, C.,  Lafuente R.A.,
\newblock Real geometric invariant theory,
\newblock preprint arXiv:1701.00643.
%
%
\bibitem{chevally} C. Chevalley,
\newblock \emph{Theorie des Groupes de Lie: Tome II, Groupes Algebriques}
\newblock Hermann $\&$ Cie Paris, 1951.
%
\bibitem{bourguignon-li-yau} J.-P. Bourguignon, P.~Li, and S.-T. Yau.
  \newblock Upper bound for the first eigenvalue of algebraic
  submanifolds.  \newblock {\em Comment. Math. Helv.}, 69(2):199--207,
  1994.
%
%
%
\bibitem{berlein-jablonki} P. Eberlein, M. Jablonski,
\newblock Closed orbit of semisimple Lie group actions and the real Hilbert-Munford functions,
\newblock New developments in Lie theory and geometry, Contemp. Math., Amer. Math. Soc., Providence, RI, 283--321, 2009.
%
\bibitem{dl} J. Der\'e, J. Lauret,
\newblock On Ricci negative solvmanifolds and their nilradicals,
\newblock  to appear in Math. Nachr.
%
\bibitem{lau1}
J. Laurret,
\newblock On the moment map for the variety
of Lie algebras,
\newblock \emph{J. Funct. Anal.} (202), 392--423, (2003)
%
\bibitem{georgula} V.~Georgulas, J.~W. Robbin, and D.~A. Salamon.
  \newblock The moment-weight inequality and the {H}ilbert-{M}umford
  criterion. https://people.math.ethz.ch/~salamon/PREPRINTS/momentweight-book.pdf
%
\bibitem{gichev-polar} V.~M. Gichev.  \newblock Polar representations
  of compact groups and convex hulls of their orbits.  \newblock {\em
    Differential Geom. Appl.}, 28(5):608--614, 2010.
%
%
\bibitem{guillemin-sternberg-convexity-1} V.~Guillemin and
  S.~Sternberg.  \newblock Convexity properties of the moment mapping.
  \newblock {\em Invent. Math.}, 67(3):491--513, 1982.
%
\bibitem{jabo} M. Jablonski,
\newblock Distinguished orbits ir reductive groups,
\newblock \emph{Rocky Mountain J. Math.} 42 (5) :1521--1549, 2012.
%
\bibitem{harish-chandra} Harish-Chandra,
\newblock Harmonic analysis on real reductive groups I. The theory of the constant term,
\newblock {\em J. Functional Analysis} 19, 104--204, (1975).
%
\bibitem{heinzner-GIT-stein} P.~Heinzner.  \newblock Geometric
invariant theory on {S}tein spaces.  \newblock {\em Math. Ann.},
  289(4):631--662, 1991.

%
%
%
%
\bibitem{heinz-stoezel} P.~Heinzner, H.~St{\"o}tzel.
\newblock{Semistable points with respect to real forms.}
\newblock {\em Math. Ann.}, 338:1--9, 2007.
%
\bibitem{heinzner-schwarz-Cartan} P.~Heinzner,  G.~W. Schwarz.
  \newblock Cartan decomposition of the moment map.  \newblock {\em
    Math. Ann.}, 337(1):197--232, 2007.
%
\bibitem{heinzner-schwarz-stoetzel} P.~Heinzner, G.~W. Schwarz,
  H.~St{\"o}tzel.  \newblock Stratifications with respect to actions
  of real reductive groups.  \newblock {\em Compos. Math.},
  144(1):163--185, 2008.
\bibitem{heinzner-stoetzel-global} P.~Heinzner, H.~St{\"o}tzel.
  \newblock Critical points of the square of the momentum map.
  \newblock In {\em Global aspects of complex geometry}, pages
  211--226.  Springer, Berlin, 2006.
%
\bibitem{heinzner-schutzdeller}
P. Heinzner, P. Sch\"utzdeller. \newblock
Convexity properties of gradient maps.
\newblock \emph{Adv. Math.},  225(3):1119--1133, 2010.
%
\bibitem{helgason}
S. Helgason. \newblock
{\em Differential Geometry, Lie Groups and Symmetric Spaces,}
\newblock Corrected reprint of the 1978 original. Graduate Studies in Mathematics, 34. American Mathematical Society, Providence, RI, 2001.
%
\bibitem{hersch} J.~Hersch.  \newblock Quatre propri\'et\'es
  isop\'erim\'etriques de membranes sph\'eriques homog\`enes.
  \newblock
%
\bibitem{kacp} V. G. Kac, D.H. Peterson,
\newblock Unitary  structure  in  representations  of  infinite-dimensional groups and a convexity theorem.
\newblock \emph{Invent. Math.} 76 (1) 1-14 (1984).
%
\bibitem{kapovich-leeb-millson-convex-JDG} M.~Kapovich, B.~Leeb,
  J.~Millson.  \newblock Convex functions on symmetric spaces, side
  lengths of polygons and the stability inequalities for weighted
  configurations at infinity.  \newblock {\em J. Differential Geom.},
  81(2):297--354, 2009.

%
\bibitem{kempf-ness} G.~Kempf, L.~Ness. \newblock {\em The length of vectors in representa
tion spaces}. \newblock In Algebraic geometry (Proc. Summer Meeting, Univ. Copenhagen, Copenha
gen, 1978), volume 732 of Lecture Notes in Math., pages 233--243. Springer, Berlin, 1979.
%
\bibitem{kirwan} F.~C. Kirwan.  \newblock {\em Cohomology of quotients
    in symplectic and algebraic geometry}, volume~31 of {\em
    Mathematical Notes}.  \newblock Princeton University Press,
  Princeton, NJ, 1984.
%
\bibitem{knapp-beyond}
A.~W. Knapp.
\newblock {\em Lie groups beyond an introduction}, volume 140 of {\em Progress
  in Mathematics}.
\newblock Birkh\"auser Boston Inc., Boston, MA, second edition, 2002.
%
\bibitem{kobayashi} S. Kobayashi, K. Nomizu,
\newblock \emph{Foundations of differential geometry Vol I},
\newblock Interscience Publishers, a division of John Wiley \& Sons, New York-London 1963
%
\bibitem{kostant-convexity} B.~Kostant.  \newblock On convexity, the
  {W}eyl group and the {I}wasawa decomposition.  \newblock {\em
    Ann. Sci. \'Ecole Norm. Sup} (4)6, 413--455, 1973.
%
\bibitem{lau1} J. Lauret,
\newblock{On the moment map for the variety of Lie algebras,}
\newblock \emph{J. Funct. Anal.} 202 (2), 392-423, (2003).
%
\bibitem{lerman} E. Lerman,
\newblock  Gradient flow of the norm squared of the moment map.
\newblock \emph{L'Enseignement Math\'emathique} 51, 117--127, (2005).
%
\bibitem{lau2} J. Lauret,
\newblock{A Canonical Compatible Metric for Geometric Structures on Nilmanifolds,}
\newblock \emph{Ann. Global Anal. Geom.} 30 (2), 107--138, (2006).
%
\bibitem{losa} S. Lojasiewicz,
\newblock Ensembles semi-analytiques,
\newblock \emph{IHES}, (1965)
%
\bibitem{neeman} A. Neeman,
\newblock The topology of quotient varieties,
\newblock \emph{Ann. of Math.} 122 (3), 419-459,  (1985).
%
\bibitem{luna-slice} D. Luna,
\newblock Slices \'etale,
\newblock  \emph{Bull. Soc. Math. France}, Mémoire 33, 81--105, 1973.
%
\bibitem{luna} D. Luna,
\newblock Sur certaines op\'erations diff\'erentiables des groupes de Lie,
\newblock \emph{Amer. J. Math.} 97. 172--181, (1975).
%
\bibitem{mostow-semisimple} G.D. Mostow,
\newblock Some new decomposition theorems for semisimple lie group.
\newblock {\em Mem. Amer. Math. Soc.}, 14:31--54,  1955.
%
\bibitem{mostow-self} G.D., Mostow,
\newblock Self-adjoint groups,
\newblock {\em Ann. of Math.} (2) 62, 44--55, (1955).
%
\bibitem{mumford-GIT} D.~Mumford, J.~Fogarty, F.~Kirwan.
  \newblock {\em Geometric invariant theory}, volume~34 of {\em
    Ergebnisse der Mathematik und ihrer Grenzgebiete (2)}.  \newblock
  Springer-Verlag, Berlin, third edition, 1994.
%
\bibitem{mundet-Crelles} I.~Mundet~i Riera.  \newblock A
  {H}itchin-{K}obayashi correspondence for {K}\"ahler fibrations.
  \newblock {\em J. Reine Angew. Math.}, 528:41--80, 2000.
%
\bibitem{mundet-Trans} I.~Mundet~i Riera.  \newblock A
  {H}ilbert-{M}umford criterion for polystability in {K}aehler
  geometry.  \newblock {\em Trans. Amer. Math. Soc.},
  362(10):5169--5187, 2010.
%
\bibitem{mundet-cont} I.~Mundet~i Riera.
\newblock Maximal weights in K\"ahler geometry: flag manifolds and Tits distance. (
With an appendix by A. H. W. Schmitt. \newblock{Contemp. Math.}, 522 "Vector bundles and complex geometry", 113--129, Amer. Math. Soc., Providence, RI, 2010.
%
\bibitem{ness} L. Ness,
\newblock Stratification of the null cone via the moment map, with an appendix by David Mumford,
\newblock \emph{Amer. J. Math.} 106, 1281--1329, (1984).
%
\bibitem{rs} R. W. Richardson, Slodowoy P.J.,
\newblock Minumun vectors for real reductive algebraic groups,
\newblock \emph{J. London Math. Soc.} (2) 42, 409--429, (1990).
%
\bibitem{schneider-convex-bodies} R.~Schneider.  \newblock {\em Convex
    bodies: the {B}runn-{M}inkowski theory}, volume~44 of {\em
    Encyclopedia of Mathematics and its Applications}.  \newblock
  Cambridge University Press, Cambridge, 1993.
%
\bibitem{schr}
A. Schrijver,
\newblock { \em Theory of linear and integer programming}, Wiley-Interscience series
in Discrete mathematics and optimization, John Wiley $\&$ Sons, Chichester (1986).
%
\bibitem{schwartz} G. Schwartz, {\em The topology of algebraic quotients}, Topological methods in algebraic tranformation groups (New Brunswick, NJ, 1988), 131-151, Prog. Math., Birkh\"auser Boston, 1989.
%
\bibitem{sjamaar} R. Sjamaar,
\newblock Holomorphic slices, symplectic reduction and multiplicity representations,
\newblock \emph{Ann. Math.} (2) 141, 87--129, (1995).
%
\bibitem{teleman-symplectic-stability} A.~Teleman.  \newblock
  Symplectic stability, analytic stability in non-algebraic complex
  geometry.  \newblock {\em Internat. J. Math.}, 15(2):183--209, 2004.
%
\bibitem{villa}
P.B. Villa,
\newblock \emph{K\"ahlerian structures of coadjoint orbits of semisimple Lie groups and their orbihedra}
\newblock Dissertation zur Erlangung desDoktorgrades der Naturwissenschaftenan der Fakult\"at f\"ur Mathematikder Ruhr-Universit\"at Bochum (2015).
%
\bibitem{wallach} N.R. Wallach,
\newblock \emph{Real Reductive groups I},
\newblock  Pure and Applied Mathematics, 132. Academic Press, Inc., Boston, MA, 1988.
%
\bibitem{whitney}
H. Whitney,
\newblock Elementary structure of real algebraic varieties
\newblock \emph{Ann. of Math.} 66, 545- 556, 1957.
%

\end{thebibliography}
\end{document}